\documentclass[reqno]{amsart}

\usepackage[utf8]{inputenc}
\usepackage[T1]{fontenc}
\usepackage[british]{babel,isodate}
\cleanlookdateon
\usepackage{lmodern}
\usepackage{enumerate}
\usepackage{tikz}
\usepackage{tikz-cd}
\usetikzlibrary{arrows}
\usepackage{cancel}
\usetikzlibrary{babel}
\usepackage{lscape}
\usepackage{url}

\usepackage{hyperref}
\hypersetup{linktocpage,colorlinks=true,citecolor=purple,linkcolor=blue,urlcolor=blue,filecolor=black}

\usepackage{float}
\usepackage{graphicx, import}
\usepackage{mathtools}
\usepackage{dsfont}
\usepackage{pinlabel}
\usepackage{yhmath}

\usepackage{xargs}                      
\usepackage{xcolor} 
\definecolor{OliveGreen}{rgb}{0,0.6,0}
 
\usepackage[colorinlistoftodos, prependcaption,textsize=tiny]{todonotes}
\newcommandx{\unsure}[2][1=]{\todo[linecolor=red,backgroundcolor=red!25,bordercolor=red,#1]{#2}}
\newcommandx{\change}[2][1=]{\todo[linecolor=blue,backgroundcolor=blue!25,bordercolor=blue,#1]{#2}}
\newcommandx{\info}[2][1=]{\todo[linecolor=OliveGreen,backgroundcolor=OliveGreen!25,bordercolor=OliveGreen,#1]{#2}}

\usepackage{stackengine}

\usepackage{amsmath,amsthm,amssymb, amsfonts, amscd}
\usepackage{rotating,subcaption}
\usepackage{url}
\usepackage{graphicx,bm}
\usepackage{mathtools}
\usepackage[mathscr]{euscript}
\allowdisplaybreaks
\usepackage{lipsum}
\usepackage{bm}
\usepackage[capitalize]{cleveref} 

\usepackage{xpatch}
\makeatletter
\AtBeginDocument{\xpatchcmd{\@thm}{\thm@headpunct{.}}{\thm@headpunct{}}{}{}}
\makeatother

\newtheorem{theorem}{Theorem}[section]
\newtheorem{proposition}[theorem]{Proposition}
\newtheorem{lemma}[theorem]{Lemma}
\newtheorem{corollary}[theorem]{Corollary}

\theoremstyle{definition}

\newtheorem{example}[theorem]{Example}
\newtheorem{examples}[theorem]{Examples}
\newtheorem{construction}[theorem]{Construction}

\theoremstyle{remark}

\newtheorem{remark}[theorem]{Remark}

\numberwithin{equation}{section}

\newcommand{\hooklongrightarrow}{\lhook\joinrel\longrightarrow}

\newcommand{\N}{\mathbb{N}}

\newcommand{\C}{\mathcal{C}}
\newcommand{\CC}{\mathsf{C}}
\newcommand{\B}{\mathsf{B}}
\newcommand{\V}{\mathcal{V}}
\newcommand{\W}{\mathcal{W}}
\newcommand{\Cstr}{\C^{\mathrm{str}}}
\newcommand{\D}{\mathcal{D}}

\newcommand{\id}{\mathrm{Id}}
\newcommand{\homC}[2]{\mathrm{Hom}_{\C}(#1,#2)}
\let\hom\relax
\newcommand{\hom}[3]{\mathrm{Hom}_{#1}(#2,#3)}
\renewcommand{\to}{\longrightarrow}

\newcommand{\rotcong}{\rotatebox[origin=c]{90}{$\cong$}}
\newcommand{\LLongrightarrow}{\equiv \mkern-4mu \Rrightarrow}

\newcommand{\toiso}{\overset{\cong}{\to}}

\usepackage{accents}
\newcommand{\uhat}{\underaccent{\check}}

\makeatletter
\newcommand{\cupr@tip}{\text{\raisebox{-0.1ex}{$\m@th\hat{}$}}}
\newcommand{\cupr}{\mathbin{\cup\cupr@}}

\newcommand{\cupr@}{%
  \mathchoice
  {\mkern-1.35mu\cupr@tip}
  {\mkern-1.35mu\cupr@tip}
  {\mkern-1.55mu\cupr@tip}
  {\mkern-1.875mu\cupr@tip}
}
\makeatother

\makeatletter
\newcommand{\capr@tip}{\text{\raisebox{0.47ex}{$\m@th\uhat{}$}}}
\newcommand{\capr}{\mathbin{\capr@\cap}}

\newcommand{\capr@}{%
  \mathchoice
  {\mkern11.6mu\capr@tip\mkern-11.6mu}
  {\mkern11.4mu\capr@tip\mkern-11.4mu}
  {\mkern11.1mu\capr@tip\mkern-11.1mu}
  {\mkern10.2mu\capr@tip\mkern-10.2mu}
}
\makeatother

\makeatletter
\newcommand{\capl@tip}{\text{\raisebox{0.47ex}{$\m@th\uhat{}$}}}
\newcommand{\capl}{\mathbin{\capl@\cap}}

\newcommand{\capl@}{%
  \mathchoice
  {\mkern2.1mu\capl@tip\mkern-2.1mu}
  {\mkern2.1mu\capl@tip\mkern-2.1mu}
  {\mkern2.3mu\capl@tip\mkern-2.3mu}
  {\mkern2.1mu\capl@tip\mkern-2.1mu}
}
\makeatother

\makeatletter
\newcommand{\cupl@tip}{\text{\raisebox{-0.1ex}{$\m@th\hat{}$}}}
\newcommand{\cupl}{\mathbin{\cupl@\cup}}

\newcommand{\cupl@}{%
  \mathchoice
  {\mkern1.35mu\cupl@tip\mkern-1.35mu}
  {\mkern1.35mu\cupl@tip\mkern-1.35mu}
  {\mkern1.55mu\cupl@tip\mkern-1.55mu}
  {\mkern1.875mu\cupl@tip\mkern-1.875mu}
}
\makeatother

\DeclareFontFamily{U}{mathx}{}
\DeclareFontShape{U}{mathx}{m}{n}{ <-> mathx10 }{}
\DeclareSymbolFont{mathx}{U}{mathx}{m}{n}
\DeclareFontSubstitution{U}{mathx}{m}{n}
\DeclareMathAccent{\widecheck}{0}{mathx}{"71}

\begin{document}


\title{Strictification and non-strictification of monoidal categories}

\date{\today}

\author{Jorge Becerra}
\address{Institut de Mathématiques de Bourgogne, UMR 5584, CNRS \& Université de Bourgogne, F-21000 Dijon, France}
\email{\href{mailto:Jorge.Becerra-Garrido@u-bourgogne.fr}{Jorge.Becerra-Garrido@u-bourgogne.fr}}
\urladdr{ \href{https://sites.google.com/view/becerra/}{https://sites.google.com/view/becerra/}} 




\begin{abstract}

In this survey paper we give account of several approaches to the strictification and non-strictification of monoidal categories, which are constructions that turn a monoidal category into a (non-)strict one monoidally equivalent to the original category, and how they are related to analogous notions in higher categorical structures. We first provide explicit, elementary models for the  (non-)strictification and show that  these two constructions give the free (non-)strict monoidal category generated by a monoidal category. Moreover, we prove in detail that these two constructions are part of a pair of free-forgetful 2-adjunctions. We later show that these constructions can be recovered from Power's general coherence theorem for 2-monads. Lastly we describe another model for the strictification based on right-module endofunctors and provide a detailed, self-contained proof that this is a particular instance of strictification of bicategories via the bicategorical analogue of the Yoneda embedding.
%
%
%
%
%
%
%
%
%
%
\end{abstract}



\maketitle

\setcounter{tocdepth}{1}
\tableofcontents


\section{Introduction}

Monoidal categories are ubiquitous in mathematics, having a far-reaching role in homotopy theory \cite{hovey,schwede,mcds}, theoretical computer science \cite{cs1,cs2}, quantum topology \cite{KRT,turaev,ghw} or theoretical physics \cite{MR1169489,runkel,kawa}, among many other areas. Monoidal categories can be thought of as the categorification of the algebraic structure of monoid. Namely, the product operation $ \cdot : M \times M \to M$ and the unit element $1 \in M$ of a monoid $M$ get categorified to a bifunctor $\otimes: \C \times \C \to \C$ and a object $\bm{1} \in \C$. The associativity axiom $(m_1 m_2)m_3 =m_1 (m_2 m_3)$ and the unit axiom $1\cdot m =m = m \cdot 1$ in a monoid can be categorified to equalities of objects $(X \otimes Y) \otimes Z = X \otimes (Y \otimes Z)$ and $1 \otimes X=X=X \otimes 1$, but in category theory this is a too strong requirement to demand. Instead, it is much more natural to ask that the objects $(X \otimes Y) \otimes Z $  and $ X \otimes (Y \otimes Z)$ are simply isomorphic, and the same applies to  $1 \otimes X$, $X$ and $X \otimes 1$. If the equalities hold, one generally talks about a \textit{strict} monoidal category.

It is a folklore result that any monoidal category is monoidally equivalent (that is, equivalent through a functor that respects the monoidal structures) to a strict one \cite{maclane,JS}. Concretely, for any monoidal category $\C$, Mac Lane \cite{CFTWM} (see also  \cite{kassel})  constructed a monoidally equivalent category $\Cstr$, called its \textit{strictification}. This result is in fact equivalent to Mac Lane's  coherence theorem. Similarly, one is sometimes also  interested in obtaining a non-strict category out of a monoidal category \cite{BNnonassociative,habiromassuyeau}. 
Tweaking Mac Lane's construction, one similarly obtains that any  monoidal category $\C$ is also monoidally equivalent to a non-strict one $\C_q$. By analogy, we call $\C_q$ the \textit{non-strictification} of $\C$.

In this paper,  we give an elementary, self-contained exposition of these two constructions (the first is well-known, the second is not). These are modelled as categories that have as objects (parenthesised) sequences of objects of the original category. In fact, we also give explicit proofs that these constructions are functorial, that is, that they give rise to 2-functors 
\begin{equation}\label{eq:str_and_q_intro}
\mathrm{str}: \mathsf{MonCat} \to \mathsf{strMonCat} \qquad , \qquad  q: \mathsf{MonCat} \to \mathsf{nonstrMonCat}.
\end{equation}



Here we have written $\mathsf{MonCat}$ (resp. $\mathsf{strMonCat}$ , $\mathsf{nonstrMonCat}$) for the 2-categories with 0-cells monoidal categories (resp. strict monoidal categories, non-strict monoidal categories), 1-cells strong monoidal functors (resp. strict monoidal functors, in both cases) and 2-cells monoidal natural transformations (in the three cases). Moreover, we also give detailed, elementary proofs of the fact that these 2-functors are in fact part of a pair of free-forgetful 2-adjunctions, 
\begin{equation}\label{eq:adjunctions_intro}
\begin{tikzcd}[column sep={4em,between origins}]
\mathsf{nonstrMonCat}    \arrow[rr, bend right, ""' pos=0.50]
& \bot &  \mathsf{MonCat} \arrow[ll, bend right,  "{q}"' pos=0.50]
  \arrow[rr, bend left, swap, "{ \mathrm{str}}"' pos=0.50] \arrow[rr, bend left, swap, ""' pos=0.50] 
& \bot &  \mathsf{strMonCat}
  \arrow[ll, bend left, swap, ""' pos=0.50]
\end{tikzcd}
\end{equation}
where the unlabelled 2-functors are the canonical forgetful ones. These results should be known to experts, but to the author's knowledge there is no literature about it where these constructions are made explicit.
%
%
%

It turns out that these two models $\Cstr$ and $\C_q$, based on (parenthesised) sequences of objects, as well as their monoidal equivalences to $\C$,   can actually be seen as two instances of a general coherence result about 2-monads due to Power \cite{power}. In the same way as there is a monad $T$ in $\mathbf{Set}$, the \textit{free monoid monad}, whose algebras are monoids, there is a 2-monad (that is, a 2-categorical version of a monad) $T$ in $\mathsf{Cat}$, the \textit{free monoid 2-monad},  whose strict algebras are essentially strict monoidal categories and whose pseudo-algebras are essentially monoidal categories. Power's coherence theorem ensures that under mild conditions (see \cref{thm:power}) every pseudo-algebra is equivalent to a strict algebra. When applying this theorem to the free monoid 2-monad, we precisely recover the monoidal equivalence between $\Cstr$ and $\C$ discussed above, see \cref{thm:recovering_equivalence_Cstr_C}. An analogous result holds for the non-strictification, using the \textit{free magma 2-monad}.

There is another approach to strictification that we would like to discuss. It is based on the fact that if $M$ is a monoid and $\mathrm{End}_M(M)$ is the monoid of self-maps of $M$ that preserve the right multiplication, $f(n\cdot m)=f(n)\cdot m$, then the ``multiplication by'' map $M \to \mathrm{End}_M(M)$, $m \mapsto m \cdot -$ is a monoid isomorphism. Categorifying the notion of monoid to that of monoidal category, mimicking this construction yields a monoidal equivalence 
\begin{equation}\label{eq:Psi_intro}
\C \overset{\simeq}{\to} \operatorname{End}_\C (\C)
\end{equation}
where $\operatorname{End}_\C (\C)$ is a strict monoidal category, called the category of \textit{right-module endofunctors} of $\C$. The objects of this category are endofunctors $F$ of $\C$ together with a natural isomorphism $F(X \otimes Y) \toiso F(X) \otimes Y$, categorifying the monoid isomorphism above, in the sense that  when a monoid is viewed as a discrete strict monoidal category, the equivalence \eqref{eq:Psi_intro} boils down to the aforementioned monoid isomorphism. 

It turns out that the above construction can be further categorified  one more time to a statement in the realm of \textit{bicategories }(also called weak 2-categories). Bicategories encode under a same concept the notions of category, monoidal category and 2-category. They can be thought of as 2-categories where the  composition law is not associative and unital ``on the nose'' but only up to natural isomorphism.

For any bicategory $\B$, there is a \textit{bicategorical Yoneda embedding} $$ Y: \B \hooklongrightarrow \mathsf{Cat}^{\B^{\mathrm{op}}}$$ which identifies $\B$ with its essential image $\mathsf{str}(\B)$, which is a 2-category and hence shows that any bicategory is equivalent (in the bicategorical sense) to a 2-category. When this identification is applied to a one-object bicategory, which happens to be the same thing as a monoidal category, then we give a detailed proof that this identification amounts exactly to the monoidal equivalence \eqref{eq:Psi_intro}. This fact is well-known to experts (e.g. \cite{leinster_hohc}), but a detailed proof does not seem to have appeared yet.

\subsection*{Organisation of the paper} 

In \cref{sec:1}, we recall the basic concepts in monoidal category theory. In \cref{sec:2},  we recall Mac Lane's construction of the strictification of a monoidal category following Kassel \cite{kassel}, where we slightly simplify the argument given therein. We also prove the universal property of the strictification (both for monoidal functors and monoidal natural transformations) and conclude with a concrete realisation of this construction for categories whose collection of objects is given by the free magma on a set. In \cref{sec:3}, in a completely analogous way to the previous one, we start constructing the non-strictification of  a monoidal category, later we show its universal property (for monoidal functors and monoidal natural transformations) and conclude with a concrete realisation for categories whose collection of objects is given by the free monoid on a set. In \cref{sec:4}, we  review 2-categories, 2-functors and 2-adjunctions,  upgrade the (non-)strictification to 2-functors  \eqref{eq:str_and_q_intro},  prove the 2-adjunctions   \eqref{eq:adjunctions_intro} and conclude obtaining a pair of adjoint equivalences for the core truncations of the aforementioned 2-categories. In \cref{sec:5}, after reviewing monads and 2-monads, we  study Power's general coherence result for 2-monads and use it to recover the constructions of $\Cstr$ and $\C_q$ when applied to the free monoid 2-monad and free magma 2-monad, respectively. Finally, in \cref{sec:6}, we start by providing a different model for the strictification of monoidal categories using the category of right-module endofunctors. Later, we gently introduce bicategories, pseudofunctors, transformations and modifications and explain how these generalise well-known concepts in (monoidal) category theory. Lastly, we consider the bicategorical Yoneda embedding and give a detailed proof that the biequivalence $\B \to \mathsf{str}(\B)$ amounts to the monoidal equivalence \eqref{eq:Psi_intro} when a monoidal category is viewed as a one-object bicategory.

\subsection*{Notation} We write ordinary categories with calligraphic or bold  font, e.g. $\C$, $\mathbf{Cat}$, and 2-categories or bicategories with sans serif font, e.g. $\mathsf{B}$, $\mathsf{C}$.

\subsection*{Acknowledgments} The author would like to thank Roland van der Veen for valuable comments on the manuscript, and to anonymous referees for pointing out the connection with Power's coherence result. Part of this paper is taken from the author's PhD thesis \textit{Universal quantum knot invariants}, written at the University of Groningen.  The author was supported by the ARN project CPJ number ANR-22-CPJ1-0001-0  at the Institut de Mathématiques de Bourgogne (IBM). The IMB receives support from the EIPHI Graduate School (contract ANR-17-EURE-0002).

\section{Monoidal categories}\label{sec:1}

In this section, we review the notions of monoidal category, monoidal functor and monoidal natural transformation. These basic definitions can be found, for instance, in \cite{turaevvirelizier}.

\subsection{Monoidal categories}\label{subsec:mon_cats}

Let $\C$ be a category. A \textit{monoidal structure} on $\C$ is the data of

\begin{enumerate}
\item a functor $$\otimes : \C \times \C \to \C,$$ called the \textit{monoidal product},
\item an object $\bm{1} \in \C$, called the \textit{unit object},
\item a natural isomorphism $$a: \otimes \circ (\otimes \times \id_\C) \overset{\cong}{\Longrightarrow} \otimes \circ (\id_\C \times \otimes  )$$ of functors $\C \times \C \times \C \to \C$, called the \textit{associativity constraint},
\item two natural isomorphisms $$ \ell: \bm{1} \otimes \id_\C \overset{\cong}{\Longrightarrow} \id_\C \qquad , \qquad   r: \id_\C  \otimes \bm{1}  \overset{\cong}{\Longrightarrow} \id_\C$$ 
of functors $\C \to \C$, called the \textit{left and right unitality constraints}, respectively,
\end{enumerate}
with the property that the following two diagrams commute for all objects $X, Y,$ $Z, M$ in $\C$:

\begin{equation}
\tag{\textit{Pentagon axiom}}
\begin{tikzpicture}[commutative diagrams/every diagram]
\node (P0) at (90:2.3cm) {$((X\otimes Y)\otimes Z)\otimes M$};    
\node (P1) at (90+72:2cm) {$(X\otimes (Y\otimes Z))\otimes M$} ;
\node (P2) at (90+2*72:2cm) {\makebox[3ex][r]{$X\otimes ((Y\otimes Z) \otimes M)$}};
\node (P3) at (90+3*72:2cm) {\makebox[3ex][l]{$X\otimes (Y\otimes (Z\otimes M))$}};
\node (P4) at (90+4*72:2cm) {$(X\otimes Y)\otimes (Z\otimes M)$};
\path[commutative diagrams/.cd, every arrow, every label]
(P0) edge node[swap] {$  a_{X,Y,Z} \otimes \id_M $} (P1)
(P1) edge node[swap] {$a_{X, Y \otimes Z, M}$} (P2)
(P2) edge node {$\id_X \otimes a_{Y,Z,M}$} (P3)
(P4) edge node {$a_{X,Y, Z \otimes M}$} (P3)
(P0) edge node {$a_{X \otimes Y, Z,M}$} (P4);
\end{tikzpicture}
\end{equation}

\begin{equation}
\tag{\textit{Triangle axiom}}
\begin{tikzcd}[column sep=0.3cm]
 (X \otimes \bm{1}) \otimes Y \drar[swap]{r_X \otimes \id_Y} \arrow{rr}{a_{X,\bm{1},Y}} &&  X \otimes (\bm{1} \otimes Y ) \dlar{\id_X\otimes\ell_{ Y}} \\
 &     X \otimes Y & 
\end{tikzcd}
\end{equation}

We call the tuple $(\C, \otimes, \bm{1}, a, \ell, r)$ a \textit{monoidal category}.

A monoidal category is said to be \textit{strict} if the associativity, left and right unit constraints are the identity natural transformations, so that $$(X \otimes Y) \otimes Z = X \otimes (Y \otimes Z) \qquad , \qquad        \bm{1} \otimes X =X \qquad , \qquad   X \otimes \bm{1} =X$$ for any objects $X,Y,Z$ in $\C$. We say that $\C$ is \textit{non-strict} if it is not strict.

Because of the Triangle axiom as well as the naturality of associativity and unit constraints, the following commutative triangles hold in any monoidal category $\C$, for any pair of objects $X,Y$ in $\C$:
\begin{equation}\label{eq:triangle2}
\begin{tikzcd}[column sep=0.25cm]
 (\bm{1} \otimes X) \otimes Y \drar[swap]{\ell_X \otimes \id_Y} \arrow{rr}{a_{\bm{1},X,Y}} &&  \bm{1} \otimes ( X \otimes Y ) \dlar{\ell_{X\otimes Y}} \\
 &     X \otimes Y & 
\end{tikzcd}
, 
\begin{tikzcd}[column sep=0.25cm]
 (  X \otimes Y )\otimes \bm{1}   \drar[swap]{r_{X\otimes Y}} \arrow{rr}{a_{X,Y,\bm{1}}} &&   X \otimes (Y \otimes \bm{1}) \dlar{\id_X \otimes r_Y} \\
 &     X \otimes Y & 
\end{tikzcd}
\end{equation}

\subsection{Monoidal functors} Let $(\C, \otimes, \bm{1})$ and $(\D, \boxtimes, \mathds{1})$ be monoidal categories. A \textit{monoidal}  functor from  $\C$ to $\D$  is
the data of
\begin{enumerate}
\item a functor $F: \C \to \D$,
\item a natural transformation $$\gamma^F =\gamma: \boxtimes \circ (F \times F) \Longrightarrow F \circ \otimes $$ of functors $\C \times \C \to \D$,
\item an arrow in $\D$  $$u^F=u: \mathds{1} \to F(\bm{1}),$$
\end{enumerate}
which are compatible with the associativity and left and right unit constraints in the following sense: for any objects $X,Y,Z$ in $\C$, the following diagrams commute:

\begin{equation}
\tag{\textit{Hexagon axiom}}
\begin{tikzcd}[column sep={1cm,between origins}, row sep={1.732050808cm,between origins}]
    & {\makebox[3ex][r]{$(F(X) \boxtimes F(Y)) \boxtimes F(Z)$}} \arrow[rr, "a'_{FX,FY,FZ}"] \arrow[ld, "\gamma_{X,Y} \boxtimes \id_{FZ}"'] && {\makebox[3ex][l]{$F(X) \boxtimes (F(Y) \boxtimes F(Z))$}} \arrow[rd, "\id_{FX} \boxtimes \gamma_{Y,Z}"] &  \\
    F(X \otimes Y) \boxtimes F(Z) \arrow[rd, "\gamma_{X \otimes Y,Z}"']&  &&  & F(X) \boxtimes F(Y \otimes Z) \arrow[dl,"\gamma_{X, Y \otimes Z}"] \\
    & {\makebox[3ex][r]{$F((X \otimes Y) \otimes Z)$}} \arrow[rr, swap, "F(a_{X,Y,Z})"'] && {\makebox[3ex][l]{$F(X \otimes (Y \otimes Z))$}}  & 
\end{tikzcd}
\end{equation}

\begin{equation}\label{eq:4}
\begin{tikzcd}[sep=2.7em]
\mathds{1} \boxtimes F(X) \rar{\ell'_{FX}} \dar[swap]{u \boxtimes \id_{FX}}  & F(X)  &  F(X) \boxtimes\mathds{1} \rar{r'_{FX}} \dar[swap]{\id_{FX} \boxtimes u}  & F(X)  \\ 
F(\bm{1}) \boxtimes F(X) \rar{\gamma_{\bm{1},X}} & F(\bm{1} \otimes X) \uar[swap]{F(\ell_X)} &   F(X)\boxtimes F(\bm{1}) \rar{\gamma_{X,\bm{1}}} & F( X \otimes \bm{1})  \uar[swap]{F(r_X)}
\end{tikzcd}
\end{equation}

\vspace*{10pt}

\noindent where we have written $(a, \ell, r)$ for the constraints of $\C$ and  $(a', \ell', r')$ for the constraints of $\D$.

The composite 
$$
\begin{tikzcd}
(\C, \otimes , \bm{1}) \rar{F} & (\D, \boxtimes, \mathds{1}) \rar{F'} & (\mathcal{E}, \odot, \mathtt{1})
\end{tikzcd}
$$
 of two monoidal functors $(F, \gamma, u) $ and $(F',  \gamma', u')$ is also monoidal with coherence constraints given by the composites
$$
\begin{tikzcd}
\mathtt{1} \rar{u'} & F'(\mathds{1}) \rar{F'(u)} & (F'F)(\bm{1})
\end{tikzcd}
$$
and 
$$
\begin{tikzcd}[column sep=1.2cm]
\odot \circ (F'F\times F'F) \arrow[Rightarrow]{r}{\gamma'(F \times F)} & F' \circ \boxtimes \circ (F \times F) \arrow[Rightarrow]{r}{F'\gamma} & F'F \circ \otimes.
\end{tikzcd}
$$

A monoidal functor $F = (F, \gamma, u): \C \to \D$ as above is \textit{strong} if $\gamma$ is a natural isomorphism of functors and $u$ is an isomorphism in $\D$. We say that $F$ is \textit{strict} if $\gamma$ is the identity natural transformation and $u$ is the identity arrow. If $F$ is not strong or strict, it is usually called \textit{lax} to distinguish it from these other two.

A \textit{monoidal equivalence} between monoidal categories $\C$ and  $\D$ is a strong monoidal functor $F: \C \to \D$ which is an equivalence of ordinary categories. If the functor $F$ is strict, then we call it a \textit{strict monoidal equivalence}.

\subsection{Monoidal natural transformations} Let $F, G: \C \to \D$ be (lax, strong or strict) monoidal functors between monoidal categories $(\C, \otimes, \bm{1})$ and $(\D, \boxtimes, \mathds{1})$. A natural transformation $\alpha: F \Longrightarrow G$ is called \textit{monoidal} if it is compatible with the monoidal constraints of $F$ and $G$ in the sense that the following two diagrams commute for all objects $X, Y$ in $\C$:
$$\begin{tikzcd}[every arrow/.append style={shift left}]
&  \mathds{1}  \arrow{rd}{u'} \arrow{ld}[swap]{u}  &         \\
    F( \bm{1})  \arrow{rr}{\alpha_{\bm{1}}} & & G(\bm{1})
\end{tikzcd} 
\qquad , \qquad
\begin{tikzcd}[every arrow/.append style={shift left}]
  F(X) \boxtimes F(Y) \arrow{r}{\gamma_{X,Y}} \arrow{d}[swap]{\alpha_{X} \boxtimes \alpha_{Y}}  &      F(X \otimes Y) \arrow{d}{\alpha_{X \otimes Y}}   \\
  G(X) \boxtimes G(Y)  \arrow{r}{\gamma'_{X,Y}} &      G(X \otimes Y)    
\end{tikzcd}$$ 
where $F=(F, \gamma, u)$ and $G=(G, \gamma', u')$. If in addition $\alpha$ is a natural isomorphism, we say that it is a \textit{monoidal natural isomorphism}.

\section{Strictification of monoidal categories}\label{sec:2}

In this section we recall Mac Lane's construction of a strict monoidal category $\Cstr$ monoidally equivalent to a given monoidal category $\C$, although we follow Kassel's approach \cite[\S XI.5]{kassel}. We also show that this is the free strict monoidal category generated by $\C$.

\subsection{Construction} \label{subsec:strictification}

Let $(\C, \otimes, \bm{1}, a, \ell, r)$ be a monoidal category. We define the category $\C^{\mathrm{str}}$ as follows: its objects are finite sequences  $S=(X_1, \ldots , X_n)$ of objects of $\C$, $n \geq 0$ (this includes the empty sequence $\emptyset$). If the \textit{parenthesisation} of a sequence $S$ is 
\begin{equation}\label{eq:Par}
Par(S) := ( \cdots (X_1 \otimes X_2) \otimes X_3) \otimes \cdots  ) \otimes X_n \in \C
\end{equation}
for any sequence $S$ of length $n >0$ and $Par(\emptyset): =\bm{1}$, define 
\begin{equation}\label{eq:1}
\hom{\C^{\mathrm{str}}}{S}{S'} := \hom{\C}{Par(S)}{Par(S')},
\end{equation}
that is, the datum of a map $f: S \to S'$ in $\Cstr$ is the same as the datum of a map $Par(f): Par(S) \to Par(S')$ in $\C$. The composite law and identities are given by those of $\C$, so that parenthesisation gives rise to a functor $Par: \Cstr \to \C$. Moreover, there is a canonical full embedding (i.e. fully faithful injective-on-objects functor) $i: \C \hooklongrightarrow \Cstr$ given by $i(X):= (X)$, the length-one sequence whose only object is $X$.
\begin{lemma}\label{lemma:1}
The canonical embedding $$i : \C \hooklongrightarrow \Cstr$$ is an equivalence of categories.
\end{lemma}
\begin{proof}
Let us see that $Par$ can be taken as a quasi-inverse of $i$. It is clear that $Par \circ i = \id_\C$. Now, for $S \in \Cstr$, let $$\delta_S : S  \to (Par (S))$$ be the unique arrow that corresponds to $\id_{Par(S)}$ under \eqref{eq:1}. These maps assemble into a natural isomorphism $\delta: \id_{\Cstr} \Longrightarrow i \circ Par$ since for an arrow $f: S  \to S'$, the equality $(i\circ Par)(f) \circ \delta_S = \delta_{S'} \circ f$ (the naturality of $\delta$) translates into $f=f$ under \eqref{eq:1}.
\end{proof}

 The category $\C^{\mathrm{str}}$ can be endowed with a strict monoidal structure, as follows: for non-empty sequences $S= (X_1, \ldots , X_n)$ and $S'=(X_{n+1}, \ldots, X_{n+m})$, set $$S*S' := (X_1, \ldots, X_{n+m}),$$ and we also put $S * \emptyset : = S =: \emptyset * S$, where $S$ is possibly empty.

To upgrade the concadenation $*$ to a functor, we must define first a family of natural isomorphisms
\begin{equation}\label{eq:2}
\theta_{S,S'} : Par(S) \otimes Par(S') \to Par (S*S')
\end{equation}
inductively on the length of $S'$. Set $\theta_{ \emptyset,S'} := \ell_{Par (S')}$ and $\theta_{S,\emptyset} := r_{Par (S)}$. For $S'=(X)$ a sequence with one object, we set $$\theta_{S,S'} := \id_{Par(S)\otimes X} : Par(S) \otimes X \to Par(S) \otimes X = Par (S * (X)) .$$ In general for $S'= \bar{S} *(X)$, then we define $\theta_{S,S'}$ as the composite
$$
\begin{tikzcd}
Par(S) \otimes Par (S') \arrow[dashed]{rr}{\theta_{S,S'}} \arrow[equals]{d} && Par (S * S') \arrow[equals]{d} \\
Par(S) \otimes (Par (\bar{S}) \otimes X)  \arrow{dr}{a^{-1}_{Par (S), Par(\bar{S}), X}}   && Par(S * \bar{S}) \otimes X  \\
 &   ( Par (S) \otimes Par(\bar{S})) \otimes X  \arrow{ur}{\theta_{S,\bar{S}}\otimes \id_X}     &
\end{tikzcd}
$$
The naturality of $\theta$ then follows from the naturality of $a, \ell$ and $r$.

Now, given arrows $f: S_1 \to S_2$ and $g: S_1' \to S_2'$ in $\C^{\mathrm{str}}$, we define the arrow $f * g : S_1 * S_1' \to S_2 * S_2'$ as the composite given by the following dashed arrow:
\begin{equation}\label{eq:3}
\begin{tikzcd}
Par(S_1) \otimes Par(S_1') \dar[swap]{f \otimes g} & Par(S_1 * S_1') \arrow{l}[swap]{\theta^{-1}_{S_1,S_1'}}  \arrow[dashed]{d}{Par(f*g)}  \\
Par(S_2) \otimes Par(S_2') \rar{\theta_{S_2, S_2'}} &  Par(S_2 * S_2') 
\end{tikzcd}
\end{equation}
The functoriality of $*$ then follows from functoriality of $\otimes$ by the above diagram. Since the concadenation $*$ is strictly associative and unital, it makes $\C^{\mathrm{str}}$ into a strict monoidal category, where the  unit object is the empty sequence $\emptyset$.

\begin{theorem}[Strictness, \cite{maclane}]\label{thm:1}
Let $\C$ be a monoidal category. The canonical embedding $$i:  \C \overset{\simeq}{\to} \Cstr$$ is a monoidal equivalence of categories.\end{theorem}
\begin{proof}
It is only left to exhibit $i$ as a strong monoidal functor. Let $u: \emptyset \to (\bm{1})$ be the unique map corresponding to $\id_{\bm{1}}$ under \eqref{eq:1}, that is, let $Par(u):=\id_{\bm{1}}$,  and define $$\eta_{X,Y} : (X)*(Y) =(X,Y) \to (X\otimes Y)$$  as $ Par(\eta_{X,Y}):=\id_{X \otimes Y}$. Cleary, $\eta$ assembles into a natural transformation $ \eta: * \circ (i \times i) \Longrightarrow i \circ \otimes$.  The Hexagon axiom and the left and right squares \eqref{eq:4} hold trivially, because their commutativity correspond under \eqref{eq:1} to the equalities $a_{X,Y,Z}=a_{X,Y,Z}$, $\ell_X =\ell_X$ and $r_X=r_X$, respectively.
\end{proof}

\begin{remark}
Our argument, that uses $i$ instead of  $Par$, simplifies the one given by Kassel, as it avoids a lengthy computation to show that $Par$ is strong monoidal \cite[XI.5.2]{kassel}, where the monoidal constraints for $Par$ are given precisely by $\theta$ and $\id_{\bm{1}}$. From our perspective that follows from generalities of monoidal categories (e.g. \cite[\S 1.4.9]{turaevvirelizier}).
\end{remark}


\begin{remark} 
Mac Lane Strictness \cref{thm:1} is equivalent to the celebrated Mac Lane Coherence Theorem, which states that in a monoidal category, any \textit{formal} diagram made out of associativity, left and right unit constraints and identities commute. Here, the word ``formal'' means that no other isomorphism or equality of objects in the category may appear. An easy argument to derive this from the Strictness theorem can be found in \cite{EGNO}.
\end{remark}

\subsection{Properties of strictification}

Let us  now discuss some properties of the previous construction. First,  we will state the universal property of the strictification, which ensures that  the strictification of a monoidal category $\C$ is the free strict  monoidal category generated by $\C$.

\begin{theorem}\label{thm:55}
Let $\C$ be a monoidal category and let $\D$ be a strict monoidal category. Given a strong monoidal functor $F: \C \to \D$, there exists a unique strict monoidal functor $\widehat{F}: \Cstr \to \D$ such that $\widehat{F} \circ i =F$,
$$  
\begin{tikzcd}
\Cstr \arrow[dashed]{rd}{\widehat{F}} & \\
\C \arrow[hook]{u}{i} \rar{F} & \D
\end{tikzcd}
$$
\end{theorem}
\begin{proof}
Let us write $(\C, \otimes, \bm{1})$ and $(\D, \boxtimes, \mathds{1})$ for the monoidal structures.  For a sequence $S=(X_1, \ldots, X_n)$, since  $\widehat{F}$ (if it exists) is strict monoidal, we must have $$\widehat{F}(S) = F(X_1) \boxtimes \cdots \boxtimes F(X_n) $$ (regardless of parentheses  as $\D$ is strict) and $\widehat{F}(\emptyset)=\mathds{1} $, hence this is the only possible definition for $\widehat{F}$ on objects. To see what $\widehat{F}$ must be in arrows, consider the natural isomorphism $\delta: \id_{\Cstr} \Longrightarrow i \circ Par$ from the proof of \cref{lemma:1}. Given an arrow $f: S \to S'$, we have $(i\circ Par)(f) \circ \delta_S = \delta_{S'} \circ f$ by the naturality of $\delta$ and applying $\widehat{F}$ to this we obtain a commutative diagram
$$
\begin{tikzcd}
[column sep=1.3cm]
\widehat{F}(S) \rar{\widehat{F}(f)} \dar[swap]{\widehat{F}(\delta_S)} & \widehat{F}(S')  \dar{\widehat{F}(\delta_{S'})} \\
F(Par(S)) \rar{F(Par(f))} & F(Par(S'))
\end{tikzcd}
$$

We will show now that for  $\widehat{F}$ as in the statement, the map $\widehat{F}(\delta_S)$ is fully determined by data of $F$, and so we will obtain a single possible definition for $\widehat{F}(f)$.

Let us now define a family of arrows $$\beta_S^F=\beta_S: \widehat{F}(S) \to F(Par(S)).$$ 
 Write $\gamma: \boxtimes \circ (F \times F)  \overset{\cong}{\Longrightarrow} F \circ \otimes $ and $u: \mathds{1} \to F(\bm{1})$ for the coherence data associated to the strong monoidal functor $F$. Inductively on the length of $S$, define $\beta_\emptyset := u$, $\beta_{(X)} := \id_{FX}$ and for $\bar{S}=S*(X)$ let $\beta_{\bar{S}}$ be the composite
$$
\begin{tikzcd}
[column sep={1.6cm}]
\widehat{F}(\bar{S})= \widehat{F}(S) \boxtimes F(X) \rar{\beta_S \boxtimes \id_{FX_n}} & F(Par(S)) \boxtimes F(X) \rar{\gamma_{Par(S), X}} &  F(Par(\bar{S}))
\end{tikzcd}
$$
Now we claim that for any such $\widehat{F}$ we must have
\begin{equation}\label{eq:44}
\beta_S = \widehat{F}(\delta_S).
\end{equation} For we observe first that $\delta_S$ can be described inductively using $\eta$, the monoidal constraint of $i$, in a similar fashion as how $\beta_S$ was defined. On the other hand, the equality of monoidal functors $\widehat{F} \circ i =F$ signifies at the level of  the monoidal constraints that 
$$
\begin{tikzcd}
F(X) \boxtimes F(Y) \dar[equals] \rar{\gamma_{X,Y}} & F(X\otimes Y) \dar[equals]\\
\widehat{F}(X,Y) \rar{\widehat{F}(\eta_{X,Y})} & \widehat{F}(i(X \otimes Y))
\end{tikzcd}
$$
which by induction implies \eqref{eq:44}.

In conclusion, if $\widehat{F}$ is as in the statement, then for any $f: S  \to S'$ in $\Cstr$ we have a commutative diagram
$$
\begin{tikzcd}
[column sep=1.3cm]
\widehat{F}(S) \rar{\widehat{F}(f)} \dar{\rotcong}[swap]{\beta_S} & \widehat{F}(S')  \dar{\beta_{S'}}[swap]{\rotcong} \\
F(Par(S)) \rar{F(Par(f))} & F(Par(S'))
\end{tikzcd}
$$
and so $\widehat{F}(f) = \beta_{S'}^{-1} \circ F(Par(f)) \circ \beta_S$ is the only possible definition for $\widehat{F}(f)$, and moreover we have $\widehat{F} \circ i = F$ as strong monoidal functors by construction.
\end{proof}

A similar result also applies to monoidal natural transformations:

\begin{proposition}\label{prop:str_nat_trans}
Let $\C$ be a monoidal category and let $\D$ be a strict monoidal category. Given a monoidal natural transformation $\alpha: F \Longrightarrow G$ between strong monoidal functors $F, G: \C \to \D$, there exists a unique monoidal natural transformation $\widehat{\alpha}: \widehat{F} \Longrightarrow \widehat{G}$ such that $\widehat{\alpha} i = \alpha$,
$$  
\begin{tikzcd}[row sep=1.6cm, column sep=1.6cm]
\Cstr \arrow[swap,bend right=10,""{name=DD}]{rd}{\widehat{G}}  \arrow[bend left=25,""{name=UU}]{rd}{\widehat{F}} & \arrow[Rightarrow, from=UU,to=DD,shorten=2pt,"\widehat{\alpha}"] \\
\C \arrow[hook]{u}{i} \rar[bend right = 20, swap, ""{name=D}]{G} \rar[bend left = 20, ""{name=U}]{F} & \D \arrow[Rightarrow, from=U,to=D,shorten=2pt,"\alpha"]
\end{tikzcd}
$$
\end{proposition}
\begin{proof}
The condition $\widehat{\alpha} i = \alpha$ means that for the length-one sequence $(X)$, we must have $\widehat{\alpha}_{(X)}= \alpha_X$. This implies that for an arbitrary sequence $S=(X_1, \ldots , X_n)$ in $\Cstr$,
$$\widehat{\alpha}_S = \widehat{\alpha}_{(X_1)* \cdots * (X_n)} = \widehat{\alpha}_{(X_1)} \boxtimes \cdots \boxtimes \widehat{\alpha}_{(X_n)} = \alpha_{X_1} \boxtimes \cdots \boxtimes  \alpha_{X_n},$$ where in the second equality we have used that $\widehat{\alpha}$ is a monoidal natural transformation between strict monoidal functors. Besides, $\widehat{\alpha}_\emptyset = \id_{\mathds{1}}$. Therefore, this is the only possible definition. It is only left to check that $\widehat{\alpha}_S :=  \alpha_{X_1} \boxtimes \cdots \boxtimes  \alpha_{X_n}$ is natural on $S$. Given an arrow $f: S \to S'$ in $\Cstr$, let us contemplate the following cube:
\begin{equation*}
\begin{tikzcd}[row sep=scriptsize,column sep=scriptsize]
&  \widehat{F}(S) \arrow[dl,"\beta_S"] \arrow[rr,"\widehat{F}(f)"] \arrow[dd,"\widehat{\alpha}_S" near end] & &   \widehat{F}(S') \arrow[dl,"\beta_{S'}"]\arrow[dd,"\widehat{\alpha}_{S'}"] \\
F(Par (S)) \arrow[rr,crossing over,"\phantom{detectiveco}F(Par(f)) "]\arrow[dd,swap,"\alpha_{Par(S)}"] & & F(Par (S')) \\
&   \widehat{G}(S) \arrow[dl, "\bar{\beta}_S"]\arrow[rr, "\widehat{G}(f) \phantom{holaaaaa}"] &  & \widehat{G}(S') \arrow[dl,"\bar{\beta}_{S'}"] \\
G(Par (S)) \arrow[rr," \phantom{ee}G(Par(f))"] & & G(Par (S'))\arrow[from=uu,"\alpha_{Par(S')}" near start,crossing over]
\end{tikzcd}
\end{equation*}
We have put $\beta= \beta^F$ and $\bar{\beta}= \beta^G$. The top and bottom faces commute by the definition of $\widehat{F}$ and $\widehat{G}$, the left and right faces commute because of the monoidality of $\alpha$, and the front face also commute by the naturality of $\alpha$. Since  $\beta_S, \beta_{S'}, \bar{ \beta}_S$ and $\bar{ \beta}_{S'}$  are isomorphisms, this implies the commutativity of the back face, which is precisely the naturality of $\widehat{\alpha}$.
\end{proof}

\subsection{A common realisation}

Many examples ``in nature'' of non-strict monoidal categories, especially in quantum topology, have as the collection of objects the free unital magma $\mathrm{Mag}(X)$ on some set $X$. We would like to make the previous construction more transparent in this case.

Suppose that  $\C$ is a  monoidal category with $\mathrm{ob}(\C) = \mathrm{Mag}(X)$ and  monoidal product given by the magma product. Let us now define a new category  $\widetilde{\C^{\mathrm{str}}}$. Its objects are given by the elements of the free monoid $\mathrm{Mon}(X)$ on $X$. Given an object $w=x_1 \cdots x_n $  of $\widetilde{\C^{\mathrm{str}}}$, its \textit{sequencing} is the object of $\C^{\mathrm{str}}$  $$ Seq (w):= (x_1 , \ldots , x_n ) \in \C^{\mathrm{str}} ,$$ with $Seq (\emptyset) := \emptyset$\footnote{Note that the first $\emptyset$ refers to the empty word in $\mathrm{Mon}(X)$ whereas the second $\emptyset$ refers to the empty sequence in $\C^{\mathrm{str}}$.}. The set of arrows in  $\widetilde{\C^{\mathrm{str}}}$ is defined as $$\hom{\widetilde{\C^{\mathrm{str}}}}{w}{w'} := \hom{\C^{\mathrm{str}}}{Seq(w)}{Seq(w')} .$$ As before, we readily see that $Seq$ extends to a fully faithful functor $$Seq: \widetilde{\C^{\mathrm{str}}} \to \C^{\mathrm{str}}.$$
We define a monoidal product on $\widetilde{\C^{\mathrm{str}}}$ as follows: on objects, it is simply given by the monoid product, $w \star w' := ww'$. Now observe that if $w=x_1 \cdots x_n $ and $w'=y_1 \cdots y_m $, then 
\begin{align*}
Seq (w \star w') &= Seq (x_1 \cdots x_n y_1 \cdots y_m) = (x_1, \ldots , x_n, y_1, \ldots , y_m) \\&=(x_1, \ldots , x_n)* (y_1, \ldots , y_m) = Seq (w)*Seq(w').
\end{align*}
This observation allows us to define the monoidal product of arrows: given $f: w_1 \to w_2$ and $g: w_1' \to w_2'$, then $f \star g: w_1 \star w_2 \to w_1' \star w_2'$ is given by $$Seq (w_1 \star w'_1) = Seq(w_1)*Seq(w_1') \overset{f*g}{\to} Seq(w_2)*Seq(w_2') = Seq (w_2 \star w_2').$$ Setting the empty word $\emptyset$ as the unit, all this data determines a strict monoidal structure on $\widetilde{\C^{\mathrm{str}}}$ such that $Seq$ is naturally a strict monoidal functor.

Let us summarise our findings:

\begin{proposition}
Let $\C$ be a  monoidal category with $\mathrm{ob}(\C) = \mathrm{Mag}(X)$ and  mo\-noidal product given by the magma product. Then there is a strict monoidal equivalence $$Seq: \widetilde{\C^{\mathrm{str}}} \overset{\simeq}{ \to} \C^{\mathrm{str}}$$
between $\C^{\mathrm{str}}$ and a  (strict) monoidal category $\widetilde{\C^{\mathrm{str}}}$ such that $\mathrm{ob}(\widetilde{\C^{\mathrm{str}}})=\mathrm{Mon}(X) $ and whose monoidal product is given by the monoid product.
\end{proposition}
\begin{proof}
It is only left to check that $Seq$ is essentially surjective. Let $S=(v_1, \ldots , v_n)$ be an element of  $\C^{\mathrm{str}}$, where $v_i \in \mathrm{Mag}(X) $. If $U: \mathrm{Mag}(X) \to \mathrm{Mon}(X)$ is the canonical map that forgets parentheses, then we will exhibit an isomorphism $Seq (U(v_1) \cdots U(v_n)) \cong S$ in $\Cstr$, or in other words, an isomorphism $$Par(Seq (U(v_1) \cdots U(v_n)) )\cong Par (S)$$ in $\C$. For this, it suffices to show that if $v \in \mathrm{Mag}(X) $ and $U(v)=x_1 \cdots x_p$, then there is an isomorphism $$\rho_v: v \toiso ( \cdots ( x_1 x_2) x_3) \cdots )x_p = Par ( Seq (U(v))) $$ in $\C$. 

Let $\gamma: \star \circ (Par \times Par) \Longrightarrow Par \circ * $ be the natural isomorphism \eqref{eq:2}, where we also denote by $\star$ the monoidal product on $\C$. We will define this isomorphism  inductively on the length $|v|$ of $v$: put $\rho_\emptyset:= \id_\emptyset$ and $\rho_x :=\id_x$. Suppose we have defined $\rho$ for objects of $\C$ of length less than $n$. Let $v \in  \mathrm{Mag}(X)$ such that $|v|=n>1$, and let $v_1, v_2 \in  \mathrm{Mag}(X)$ the unique elements of positive length such that $v=v_1 v_2$. Then define $\rho_v$ as the composite
$$\begin{tikzcd}[row sep={0.2cm}, column sep={2.5cm}]
v= v_1 \star v_2 \rar{\rho_{v_1}\star \rho_{v_2}} &  Par ( Seq (U(v_1))) \star Par ( Seq (U(v_2)))\\
\phantom{v= v_1 \star v_2} \arrow{r}{\gamma_{Seq(U(v_1)),Seq(U(v_2))}} &   Par \big( Seq(U(v_1)) * Seq (U(v_2))    \big)   \\
\phantom{v= v_1 \star v_2 } \arrow[equals]{r} & Par ( Seq (U(v))) 
\end{tikzcd}$$
Since $\gamma$ is a natural isomorphism, $\rho_v$ is an isomorphism as well.
 \end{proof}

\section{Non-strictification of monoidal categories}\label{sec:3}

The goal of this section is to show, using an analogous  construction to the one in the previous section, that any monoidal category is monoidally equivalent to a non-strict one. Our construction is inspired by \cite{habiromassuyeau}. Later we will similarly prove the universal property of this construction.

\subsection{Main construction}\label{subseq:non-strict}

Let $\mathrm{Mag}(\bullet)$ be the free unital magma generated by the singleton $\bullet$. Recall that elements of $\mathrm{Mag}(\bullet)$ are parenthesised sequences of bullets (included the empty sequence), e.g. $((\bullet\bullet)\bullet)(\bullet\bullet)$. Forgetting parenthesis gives a map $| - |: \mathrm{Mag}(\bullet) \to \N$ that counts the number of bullets.

Given a monoidal category $(\C, \otimes, \bm{1}, a, \ell, r)$, we define $\C_q$ as follows: its objects are pairs $(S, t)$, where $S$ is a finite sequence of objects of $\C$, $t \in \mathrm{Mag}(\bullet)$ and $|t| = |S|$, where $|S|$ is the length of $S$. For such a pair  $(S, t)$, define its \textit{parenthesisation} $Par(S, t) \in \C$ as the object in $\C$ given by inserting the $i$-th element of $S$ in the $i$-th bullet of $t$ and inserting tensor products between them, e.g. $$ Par((X_1, \ldots , X_5), ((\bullet\bullet)\bullet)(\bullet\bullet))  =  ((X_1 \otimes  X_2)\otimes X_3) \otimes (X_4 \otimes X_5),  $$ and $Par(\emptyset, \emptyset) := \bm{1}$. More precisely, $Par(S,t)$  is defined by the following inductive rules:
\begin{enumerate}
\item $Par(\emptyset, \emptyset) := \bm{1}$,
\item $Par((X), \bullet):= X$,
\item Given $(S,t) \in \C_q$ with $|t|>1$, there exist unique $t_1,t_2 \in \mathrm{Mag}(\bullet)$  with $|t_1|, |t_2| \geq 1$, $|t_1|+ |t_2|= |t|$ such that $t=t_1 t_2$. If $S_1, S_2$ are the unique subsequences of $S$ of lengths $|t_1|, |t_2|$ respectively such that $S=S_1 * S_2$, then define $$Par(S,t) := Par (S_1, t_1) \otimes Par (S_2,t_2).$$
\end{enumerate}

The set of morphisms in $\C_q$ is given by
\begin{equation}\label{eq:arrows_in_Cq}
\hom{\C_q}{(S,t)}{(S',t')} := \hom{\C}{Par (S,t)}{Par (S',t')}
\end{equation}
with the composite and identities determined by those of $\C$. Once again, the datum of a map $f: (S,t) \to (S',t')$ in $\C_q$ is the same as the datum of a map $Par(f): Par(S,t) \to Par(S',t')$  in $\C$. This means that the parenthesisation gives rise to a functor $Par: \C_q \to \C$. On the other hand, note that there is a canonical full embedding $j: \C \hooklongrightarrow \C_q$ defined by $j(X) :=((X), \bullet)$.

\begin{lemma}\label{lemma:007}
The canonical embedding $$j: \C \hooklongrightarrow \C_q$$ is an equivalence of categories.
\end{lemma}
\begin{proof}
We will prove that $Par$ is a quasi-inverse of $j$. On the one hand, we have $Par \circ j = \id_\C$. On the other hand, for $(S,t) \in \C_q$, define $$\delta_{(S,t)}: (S,t) \to (Par(S,t), \bullet)$$ as $Par(\delta_{(S,t)}):= \id_{Par(S,t)}$. As in \cref{lemma:1}, these maps trivially assemble into a natural transformation $\delta: \id_{\C_q} \Longrightarrow j \circ Par$, which gives the result.
\end{proof}


Let us now endow $\C_q$ with a canonical non-strict monoidal structure. Given objects $(S,t), (S',t')$ in $\C_q$, set $$(S,t) *(S',t') := (S * S', tt').$$ We also agree that $(S,t)* (\emptyset, \emptyset) := (S,t) =: (\emptyset,\emptyset) * (S,t)$. We immediately see from the definition of $Par$ that
\begin{equation}\label{eq:gamma_id}
Par((S,t) *(S',t')) = Par (S * S', tt') = Par(S,t) \otimes Par(S',t').
\end{equation}

We can now define $*$ on arrows directly: given  $f: (S_1, t_1) \to (S_2,t_2)$ and $g: (S_1',t_1') \to (S_2',t_2)$, we define $f*g$ as $ Par(f*g):= f \otimes g$. More precisely, the arrow $f *g : (S_1, t_1)* (S_1',t_1') \to (S_2, t_2)* (S_2',t_2')$ is determined by the composite
$$\begin{tikzcd}[row sep={0.2cm}]
Par ((S_1, t_1)* (S_1',t_1')) \arrow[equals]{r} &   Par(S_1, t_1) \otimes Par(S'_1, t'_1)\\
\phantom{Par ((S_1, t_1)* (S_1',t_1'))} \rar{f \otimes g} &  Par(S_2, t_2) \otimes Par(S'_2, t'_2)\\
\phantom{Par ((S_1, w_1)* (S_1',w_1'))} \arrow[equals]{r} & Par ((S_2, t_2)* (S_2',t_2'))
\end{tikzcd}$$
The functoriality of $*$ now follows directly from that of $\otimes$. The unit for this monoidal product is given by $(\emptyset, \emptyset)$. We define the left and right unitality constraints as the identity natural isomorphisms, so that $*$ is strictly left and right unital. The associativity constraint $a_q$ 
$$\begin{tikzcd}
 \left[ (S,t)* (S',t') \right] * (S'',t'') \arrow[equals]{d} \rar{a_q} &  (S,t)*  \left[(S',t')* (S'',t'') \right] \arrow[equals]{d} \\
(S*S'*S'' , (tt')t'') & (S*S'*S'' , t(t't''))
\end{tikzcd}$$
is defined as $Par(a_q):=a_{Par(S,t), Par(S',t'), Par(S'',t'')}$. All these elements make $\C_q$ into a monoidal category: the Triangle axiom holds trivially and the Pentagon axiom follows from that for $\C$. Note that this is a non-strict monoidal category, even if $\C$ is strict, since $(tt')t''$ and $t(t't'')$ are in general different elements in $\mathrm{Mag}(\bullet)$.



\begin{theorem}\
The previous  non-strict monoidal structure on $\C_q$ makes $$j: \C \overset{\simeq}{\to} \C_q$$  a monoidal equivalence of categories.
\end{theorem}
\begin{proof}
It suffices to exhibit $j$ as a strong monoidal functor. Let $u:(\emptyset , \emptyset) \to ((\bm{1}), \bullet)$ be given by $Par(u):= \id_{\bm{1}}$, and set $$\eta_{X,Y}: ((X), \bullet) * ((Y),\bullet) = ((X,Y), \bullet  \bullet) \to ((X \otimes Y), \bullet)$$ to be determined by $Par(\eta_{X,Y}):=\id_{X \otimes Y}$.  Trivially, $\eta$ assembles into a natural transformation $ \eta: * \circ (i \times i) \Longrightarrow i \circ \otimes$.  As in \cref{thm:1}, the commutativity of the  Hexagon axiom and the left and right squares \eqref{eq:4} correspond under \eqref{eq:1} to the equalities $a_{X,Y,Z}=a_{X,Y,Z}$, $\ell_X =\ell_X$ and $r_X=r_X$, respectively, so they trivially hold.
\end{proof}

\subsection{Properties of non-strictification}

We now state the main properties of the previous construction, that will say that the non-strictification is the free non-strict monoidal category generated by a monoidal category, in a completely analogous manner as \cref{thm:55} and \cref{prop:str_nat_trans}.

\begin{theorem}\label{thm:unique_nonstr}
Let $\C, \D$ be  monoidal categories. Given a strong monoidal functor $F: \C \to \D$, there exists a unique strict monoidal functor $\widehat{F}: \C_q \to \D$ such that $\widehat{F} \circ j =F$,
$$  
\begin{tikzcd}
\C_q \arrow[dashed]{rd}{\widehat{F}} & \\
\C \arrow[hook]{u}{j} \rar{F} & \D
\end{tikzcd}
$$
\end{theorem}
\begin{proof}
Let us put $(\C, \otimes, \bm{1})$ and $(\D, \boxtimes, \mathds{1})$ for the monoidal structures. If such $\widehat{F}$ exists, then we must have $$ \widehat{F} (\emptyset, \emptyset)  = \mathds{1} \quad , \quad \widehat{F} ((X), \bullet)= F(X) \quad , \quad \widehat{F} ((S_1, t_1)*(S_2, t_2)) = \widehat{F} (S_1, t_1) \boxtimes \widehat{F} (S_2, t_2) $$
so since every object $(S,t)$ with $|S|>1$ can be uniquely written as $(S,t) = (S_1, t_1)*(S_2, t_2)$ with $|S_i|>0$,  the previous equalities inductively determine the only possible definition of $\widehat{F}$ on objects.

To see what $\widehat{F}$ must be on arrows, we argue closely to \cref{thm:55}: if $\delta: \id_{\C_q} \Longrightarrow j \circ Par$ is the natural isomorphism of \cref{lemma:007}, then applying $\widehat{F}$ to the naturality of $\delta$ yields the commutative diagram
$$
\begin{tikzcd}
[column sep=1.3cm]
\widehat{F}(S,t) \rar{\widehat{F}(f)} \dar[swap]{\widehat{F}(\delta_{(S,t)})} & \widehat{F}(S',t')  \dar{\widehat{F}(\delta_{(S',t')})} \\
F(Par(S,t)) \rar{F(Par(f))} & F(Par(S',t'))
\end{tikzcd}
$$
for any arrow $f: (S,t) \to (S',t')$ in $\C_q$. To see that $\widehat{F}(\delta_{(S,t)})$ only depends on $F$, we define a family of  maps $$\beta_{(S,t)}^F=\beta_{(S,t)}: \widehat{F}(S,t) \to F(Par(S,t))$$ inductively on the length as follows: if $\gamma$ and $u$ are the  coherence constraints of $F$, set $\beta_{(\emptyset, \emptyset)} := u$, $\beta_{((X), \bullet)} := \id_{FX}$ and for $(S,t)= (S_1, t_1)*(S_2, t_2)$, let $\beta_{(S,t)}$ be the composite
$$
\begin{tikzcd}[column sep=2.2cm, row sep=0.2cm]
\widehat{F}(S,t) = \widehat{F}(S_1,t_1) \boxtimes \widehat{F}(S_2,t_2) \rar{\beta_{(S_1,t_1)} \boxtimes \beta_{(S_2,t_2)} } & F(Par(S_1,t_1)) \boxtimes F(Par(S_2,t_2))\\
\phantom{\widehat{F}(S,t) = \widehat{F}(S_1,t_1) \boxtimes \widehat{F}(S_2,t_2)} \rar{\gamma_{Par(S_1, t_1), Par(S_2, t_2)}} & F(Par(S_1,t_1) \otimes Par(S_2,t_2)  ) \\
\phantom{\widehat{F}(S,t) = \widehat{F}(S_1,t_1) \boxtimes \widehat{F}(S_2,t_2)} \rar[equals] & F(Par(S,t))
\end{tikzcd}
$$
An argument identical to that given in \cref{thm:55} shows that $$\beta_{(S,t)} = \widehat{F}(\delta_{(S,t)}),$$ hence $\widehat{F}(f) = \beta_{(S',t')}^{-1} \circ  F(Par (f)) \circ \beta_{(S,t)}$ is the only possible definition for $\widehat{F}(f)$. Furthermore, we have $\widehat{F} \circ j = F$ as strong monoidal functors by construction.
\end{proof}

\begin{proposition}\label{prop:nonstr_nat_trans}
Let $\C$ be a monoidal category and let $\D$ be a non-strict monoidal category. Given a monoidal natural transformation $\alpha: F \Longrightarrow G$ between strong monoidal functors $F, G: \C \to \D$, there exists a unique monoidal natural transformation $\widehat{\alpha}: \widehat{F} \Longrightarrow \widehat{G}$ such that $\widehat{\alpha} j = \alpha$,
$$  
\begin{tikzcd}[row sep=1.6cm, column sep=1.6cm]
\C_q \arrow[swap,bend right=10,""{name=DD}]{rd}{\widehat{G}}  \arrow[bend left=25,""{name=UU}]{rd}{\widehat{F}} & \arrow[Rightarrow, from=UU,to=DD,shorten=2pt,"\widehat{\alpha}"] \\
\C \arrow[hook]{u}{j} \rar[bend right = 20, swap, ""{name=D}]{G} \rar[bend left = 20, ""{name=U}]{F} & \D \arrow[Rightarrow, from=U,to=D,shorten=2pt,"\alpha"]
\end{tikzcd}
$$
\end{proposition}
\begin{proof}
If such $\widehat{\alpha}$ exists, then it must satisfy
$$ \widehat{\alpha}_{(\emptyset, \emptyset)} = \id_{\mathds{1}}  \quad , \quad \widehat{\alpha}_{((X), \bullet)} = \alpha_{X} \quad , \quad \widehat{\alpha}_{(S,t)*(S',t')} = \widehat{\alpha}_{(S,t)} \boxtimes \widehat{\alpha}_{(S',t')},    $$ and as in the previous arguments these rules inductively determine the value $\widehat{\alpha}_{(S,t)}$ for any object $(S,t)$, so this is the only possible definition. The naturality of $\widehat{\alpha}$ on $(S,t)$ follows as in \cref{prop:str_nat_trans} considering a similar cube and arguing that it must be commutative.
\end{proof}

\subsection{Common realisations}

In nature one finds many examples of strict monoidal categories whose collection of objects is given by the free monoid $\mathrm{Mon}(X)$ on some set $X$. As we did with the strictification, we would like to make the non-strictification more concrete for this particular case.

Suppose that $\C$ is a monoidal category with $\mathrm{ob}(\C) = \mathrm{Mon}(X)$ for some set $X$ and whose monoidal product is given by the monoid product. Define a new category $\widetilde{\C_q}$ as follows: its objects are given by the elements of the free magma $\mathrm{Mag}(X)$ on $X$.  Write $U: \mathrm{Mag}(X) \to \mathrm{Mon}(X)$ for  the canonical map that forgets parentheses and $p: \mathrm{Mag}(X) \to \mathrm{Mag}(\bullet)$ for the magma map induced by the unique map of sets $X \to \bullet$. For $v \in \mathrm{Mag}(X)$, define its \textit{sequencing} as $$Seq (v) := ( (x_1, \ldots , x_n), p(v)) \in \C_q$$ where $U(v)= x_1 \cdots x_n$, and set $Seq (\emptyset) := (\emptyset, \emptyset)$. The hom sets in $\widetilde{\C_q}$ are defined as
$$ \hom{\widetilde{\C_q}}{v}{v'} := \hom{\C_q}{Seq (v)}{Seq (v')} , $$ and the identities and composite law are defined as those in $\C_q$. As before, this gives rise to a fully faithful functor $$Seq: \widetilde{\C_q} \to \C_q.$$
Let us now define a monoidal structure on $\widetilde{\C_q}$: on objects, it is given by the magma product, $v \star v' := vv'$. Now, if $v,v' \in \mathrm{Mag}(X)$, $U(v)= x_1 \cdots x_n$ and $U(v')= y_1 \cdots y_m$, then note that 
\begin{align*}
Seq (v \star v') &= \big(  (x_1, \ldots , x_n, y_1, \ldots, y_m)   ,p(vv')   \big) \\ &=  \big(  (x_1, \ldots , x_n)*( y_1, \ldots, y_m)   ,p(v)p(v')   \big) = Seq (v)* Seq(v')
\end{align*}
This allows us to define the monoidal product $f  \star g$ of morphisms $f: v_1 \to v_2$ and $g: v_1' \to v_2'$ as the composite $$Seq (v_1 \star v'_1) = Seq(v_1)*Seq(v_1') \overset{f*g}{\to} Seq(v_2)*Seq(v_2') = Seq (v_2 \star v_2').$$
Once more, setting the empty word $\emptyset$ as the unit, all this data determines a non-strict monoidal structure on $\widetilde{\C_q}$ such that $Seq$ is naturally a strict monoidal functor.

As before, this implies

\begin{proposition}
Let $\C$ be a  monoidal category with $\mathrm{ob}(\C) = \mathrm{Mon}(X)$ and  mo\-noidal product given by the monoid product. Then there is a strict monoidal equivalence $$Seq: \widetilde{\C_q} \overset{\simeq}{ \to} \C_q$$
between $\C_q$ and a  (non-strict) monoidal category $\widetilde{\C_q}$ such that $\mathrm{ob}(\widetilde{\C_q})=\mathrm{Mag}(X) $ and whose monoidal product is given by the magma product.
\end{proposition}
\begin{proof}
Again it is left to check that $Seq$ is essentially surjective also in this case. Given an object $(S,t) \in \C_q$, with $S=(w_1, \ldots , w_n)$ and $t \in \mathrm{Mag}(\bullet)$, let $v \in \mathrm{Mag}(X) $ be an arbitrary parenthesisation of the product $w_1 \cdots w_n \in \mathrm{Mon}(X)$. Then since the monoidal product of $\C$ on objects is given by concatenation  we have $$Par(S,t)= w_1 \cdots w_n =Par(Seq(v)),$$ which exhibits an isomorphism $(S,t) \cong Seq (v)$ in $\C_q$.
\end{proof}

The previous result recovers the notion of non-strictification given in \cite[\S 3.3]{habiromassuyeau}. We would like to remark that in this publication, the non-strictification is only defined for strict monoidal categories whose monoid object $\mathrm{ob}(\C)$ is free. Our construction of non-strictification can be therefore viewed as a generalisation of that to arbitrary monoidal categories (strict or not).

\subsection{Main motivation}

Before concluding this section we would like to give some context where the non-strictification of a monoidal category is crucial. This digression is independent of the rest of the paper.

\begin{example}
Let us write $\mathrm{Mon}(+,-)$ for the free monoid on the two-element set $\{ +,- \}$, and let $s,t \in \mathrm{Mon}(+,-)$. A \textit{tangle with source $s$ and target $t$} is a finite system of oriented, disjoint embedded arcs and circles in the strip $\mathbb{R}^2 \times [0,1]$ such that the circle components lie in $\mathbb{R}^2 \times (0,1)$ and the endpoints of the arcs are the points $(1,0,0), (2,0,0), \ldots , (|s|,0,0)$ and $(1,0,1),(2,0,1), \ldots , (|t|,0,1)$. A tangle $T$ is said to be \textit{framed} if it is endowed with a non-singular normal vector field which at the endpoints of the arcs equals the vector $(1,0,0)$. Below we have depicted a framed tangle with source $+-+$ and target $+-+$:

\begin{equation*} 
\centering
\includegraphics[width=0.25\textwidth]{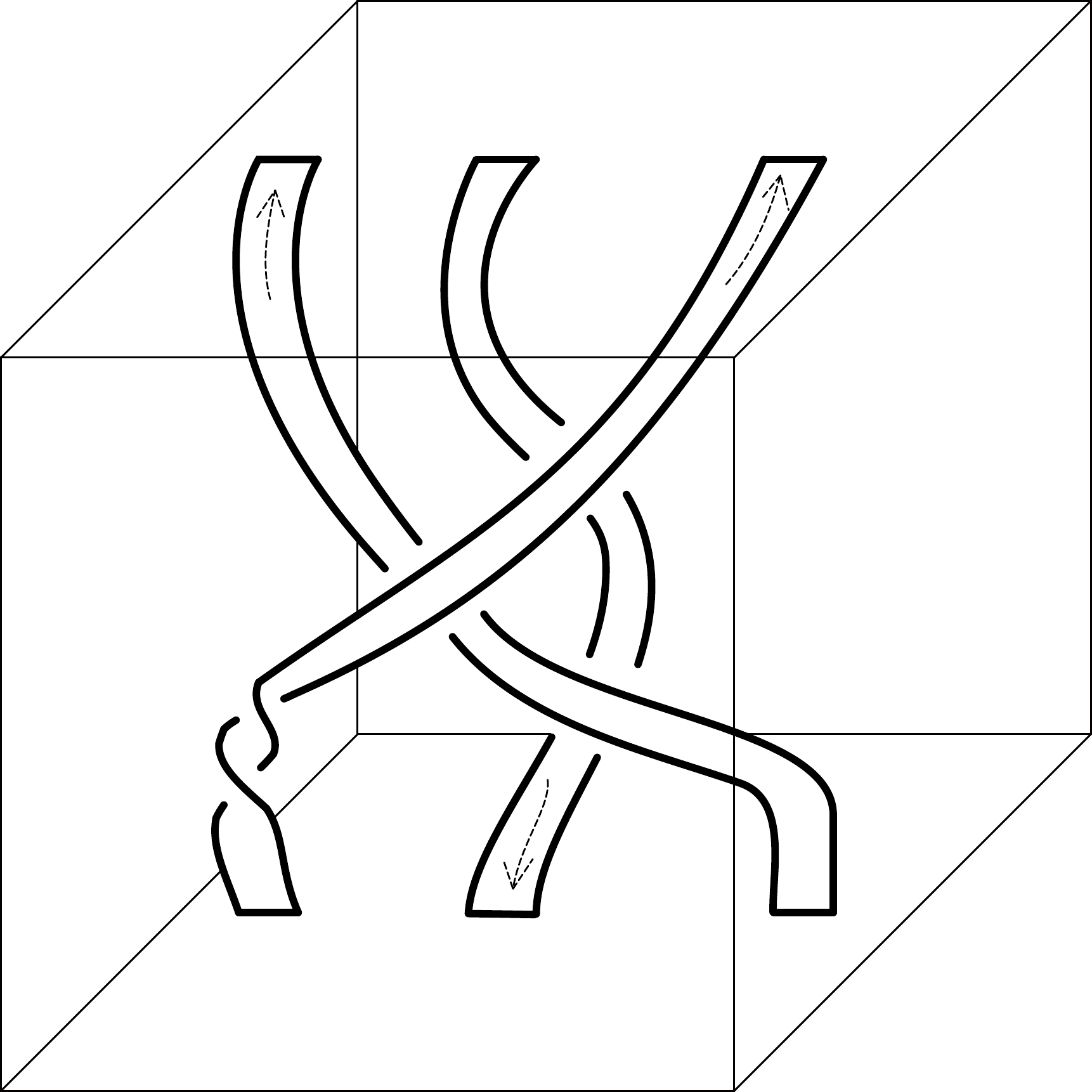}
\end{equation*}
\vspace{0.1cm}\noindent

Two framed tangles with the same source and target are \textit{isotopic} if one can be smoothly deformed into the other within the class of framed tangles with the given source and target. Observe that (framed) knots and links are contained in this framework.

Framed tangles can be organised into a strict monoidal category  $\mathcal{T}$ as follows. The objects are given by the elements of the set $\mathrm{Mon}(+,-)$. An arrow from $s$ to $t$ is an isotopy class of a framed tangle with source $s$ and target $t$. The composition $T_2 \circ T_1$ is obtained by stacking $T_2$ on top of $T_1$ and compressing the result into $\mathbb{R}^2 \times [0,1]$. The identity of an object $w \in \mathrm{Mon}(+,-)$ consists of $|w|$ disjoint vertical arcs with constant framing and orientations determined by $w$.

The category $\mathcal{T}$ is endowed with a canonical monoidal structure. On objects, the monoidal product consists of the monoid product. Given tangles $T,T'$, the monoidal product $T \otimes T'$ is the juxtaposition obtained by placing $T'$ to the right of $T$. The unit of the monoidal stucture is given by the empty word.
\end{example}

The \textit{Kontsevich invariant} \cite{kontsevich} is a very strong graph-valued invariant of framed knots. In fact, it is expected that this invariant classifies knots \cite{ohtsukiproblems}. 

The Kontsevich invariant can be promoted to a strict monoidal functor \cite{BNnonassociative,lemurakami}
$$Z_{\varphi}: \widetilde{{\mathcal{T}}_q} \to \mathcal{A}$$
which extends the invariant to \textit{$q$-tangles} or arrows of the non-strictification $\widetilde{{\mathcal{T}}_q}$. The  category $\mathcal{A}$ is the so-called category of ``Jacobi diagrams in polarised 1-manifolds''. This functor depends on a choice of an infinite power series $\varphi \in \mathbb{Q} \langle \langle X,Y \rangle \rangle$ in two non-commuting variables, called a \textit{Drinfeld series}. The functor $Z_{\varphi}$ is only well-defined on $\widetilde{{\mathcal{T}}_q}$, but not on $\mathcal{T}$. This invariant is related to many other of the so-called quantum knot invariants, see e.g. \cite{becerra_thesis}.

\section{Higher categorical perspective}\label{sec:4}

In this section, we will upgrade the strictification and non-strictification to the 2-categorical level. Throughout this section we will denote as $*$ the discrete category with a single object.

\subsection{2-categories} \label{subsec:2-categories}  A \textit{2-category} $\mathsf{C}$  is the data of
\begin{enumerate}
\item a family of objects $A, B, C, \ldots$ of $\mathsf{C}$,
\item for every pair of objects $A, B \in \mathsf{C}$, a hom-category $\mathsf{C}(A,B)$,
\item for each $A \in \mathsf{C}$, a functor $\id_{A}: * \to \mathsf{C}(A,A)$,
\item for every triple $A,B,C \in \mathsf{C}$, a composition functor (called the \textit{horizontal composition}) $$\circ_{A,B,C}=\circ : \mathsf{C}(B,C) \times \mathsf{C}(A,B) \to \mathsf{C}(A,C)$$
\end{enumerate}
satisfying the following associativity and unital conditions: for every $A,B,C,D \in \mathsf{C}$, we have
\begin{equation}\label{eq:assoc_2-cat}
\circ_{A,B,D}(\circ_{B,C,D} \times \id_{\mathsf{C}(A,B)}) =\circ_{A,C,D}(\id_{\mathsf{C}(C,D)} \times \circ_{A,B,C}) 
\end{equation}
as functors $ \mathsf{C}(C,D) \times \mathsf{C}(B,C) \times \mathsf{C}(A,B) \to \mathsf{C}(A,D),$ and
\begin{align}
& \circ_{A,B,B}(\id_{B} \times \id_{\mathsf{C}(A,B)}) = \id_{\mathsf{C}(A,B)} \\
& \circ_{A,A,B}(\id_{\mathsf{C}(A,B)}\ \times \id_{A}) = \id_{\mathsf{C}(A,B)} \label{eq:rightunit_2-cat}
\end{align}
as functors $\mathsf{C}(A,B) \to \mathsf{C}(A,B)$.

If $\mathsf{C}$ is a 2-category, then the objects $A,B,C, \ldots$ are typically called \textit{0-cells}, the objects of the hom-categories $\mathsf{C}(A,B)$ \textit{1-cells}, and the arrows of these hom-categories  $\mathsf{C}(A,B)$ \textit{2-cells}.

A 2-category is the same as a category enriched over the symmetric monoidal category $(\mathbf{Cat}, \times, *)$ of categories and functors \cite[\S 3]{riehl2014}. Note that a 2-category $\mathsf{C}$ with a single object $\star$ is the same data as an ordinary category $\C := \mathsf{C}(\star, \star)$. More generally, every  2-category $ \mathsf{C}$ has an underlying ordinary category $\mathsf{C}_{0}$ whose objects are those of $\mathsf{C}$ and whose arrows are given  by $$\hom{\mathsf{C}_{0}}{A}{B} := \hom{\mathbf{Cat}}{*}{\mathsf{C}(A,B)},$$ that is, discarding  2-cells.

Notice that in a 2-category, we have the notion of 1-cells and 2-cells being isomorphims, that we call them \textit{ 1-isomorphisms} and \textit{2-isomorphisms}. Besides, we also have the notion of  equivalence of objects: we say that a 1-cell $f: A  \to B$ is an \textit{equivalence} if there exists another 1-cell $g: B \to A$ and two 2-isomorphisms $\alpha_1: g \circ f \to \id_A$ and  $ \alpha_2: f \circ g \to \id_B$, where $\id_C$ of an object $C$ is the 1-cell image of the functor $1_C$.

\subsection{2-functors}\label{subsec:2-functor}

Let  $\mathsf{C}, \mathsf{C'}$ be 2-categories. A \textit{2-functor} $F: \mathsf{C}  \to \mathsf{C'}$  is the data of
\begin{enumerate}
\item for every object $A \in \mathsf{C}$, an object $F(A) \in \mathsf{C'}$,
\item for every pair of objects $A, B \in \mathsf{C}$, a functor $$F_{A,B}: \mathsf{C}(A,B) \to \mathsf{C'}(F(A),F(B))$$
\end{enumerate}
such that for every triple $A,B,C \in \mathsf{C}$ we have
\begin{equation}\label{eq:Anotherdiagram1}
\circ_{FA,FB,FC}(F_{B,C} \times F_{A,B}) = F_{A,C} \circ_{A,B,C} 
\end{equation}
as functors $\mathsf{C} (B,C) \times \mathsf{C}(A,B) \to \mathsf{C}'(FA,FC)  $, and
\begin{equation}\label{eq:Anotherdiagram2}
 \id_{FA} = F_{A,A} \id_{A}
\end{equation}
 as functors $ * \to \mathsf{C}' (FA,FA)$.

Observe that 2-categories and 2-functors form an ordinary category, that we denote by $2 {\text -} \mathbf{Cat} $.

%
%
%

%
%
%
%

\subsection{2-adjunctions} Let $\mathsf{C}, \mathsf{C}'$ be 2-categories.  A \textit{2-adjunction} is a pair $F: \mathsf{C} \to \mathsf{C}'$ and $G: \mathsf{C}' \to \mathsf{C}$ of 2-functors together with a family of isomorphisms of categories   $$\varphi_{A,A'}: \mathsf{C}'(FA, A') \toiso \mathsf{C}(A, GA')$$ for every $A \in \mathsf{C}$ and $A' \in \mathsf{C}'$, which are natural in the sense  that for every $A,B  \in \mathsf{C}$ and $A',B' \in \mathsf{C}'$ the following two diagrams  of categories and functors commute:

\begin{equation*}
\begin{tikzcd}
\mathsf{C}'(FA,A') \times \mathsf{C}(B,A) \arrow{rr}{\id \times F_{B,A}}  \dar[swap]{\varphi_{A,A'}\times  \id} & & \mathsf{C}'(FA,A') \times \mathsf{C}'(FB,FA)   \dar{ \circ}\\
\mathsf{C}(A, GA') \times \mathsf{C}(B,A) \arrow{dr}{\circ} & & \mathsf{C}'(FB,A') \arrow{dl}[swap]{\varphi_{B,B'}} \\
& \mathsf{C}(B,GA') &
\end{tikzcd}
\end{equation*}

$$
\begin{tikzcd}[column sep=1.9cm]
\mathsf{C}'(A',B') \times \mathsf{C}' (FA,A') \rar{G_{A',B'} \times \varphi_{A,A'}}  \dar{\circ} & \mathsf{C}(GA', GB') \times \mathsf{C}(A, GA') \dar{\circ} \\
\mathsf{C}'(FA, B') \rar{\varphi_{A,B'}} & \mathsf{C}(A, GB')
\end{tikzcd}
$$

\subsection{Strictification as a 2-functor}

Let $\mathsf{MonCat}$ be the 2-category whose 0-cells are monoidal categories, whose 1-cells are strong monoidal functors and whose 2-cells are monoidal natural transformations. The composition of 2-cells inside a hom-category is given by vertical composition of the natural isomorphisms, whereas horizontal composition for 2-cells is given by horizontal composition of natural transformations. Similarly, we let $\mathsf{strMonCat}$ (resp. $\mathsf{nonstrMonCat}$) be the 2-categories of strict (resp. non-strict) monoidal categories as 0-cells, strict monoidal functors as 1-cells and monoidal natural transformations as 2-cells.

Our aim is to upgrade the strictification of monoidal categories that we constructed in \cref{subsec:strictification} to a 2-functor 
$$\mathrm{str}: \mathsf{MonCat} \to \mathsf{strMonCat},$$ as follows: given a monoidal category $\C$, let $\mathrm{str}(\C):= \Cstr$. Given a strong monoidal functor $F: \C \to \D$ between monoidal categories, define $\mathrm{str}(F)=F^{\mathrm{str}}$ as the unique strict monoidal functor $F^{\mathrm{str}}:  \Cstr \to \D^{\mathrm{str}}$ making the diagram
$$  
\begin{tikzcd}
\Cstr \arrow[dashed]{rrd}{F^{\mathrm{str}}} & & \\
\C \arrow[hook]{u}{i} \rar{F} & \D \arrow[hook]{r}{i} & \D^{\mathrm{str}}
\end{tikzcd}
$$
commute according to \cref{thm:55}. It is readily seen that for $S=(X_1, \ldots , X_n)$ in $\C$, we have $$ F^{\mathrm{str}}(S) = (FX_1, \ldots , FX_n)$$ and $F^{\mathrm{str}}(\emptyset)= \emptyset$; and  given an arrow $f: S \to S'$ in $\C$, we have that $F^{\mathrm{str}}(f)$ is the following   dashed arrow,
\begin{equation}\label{eq:def66}
\begin{tikzcd}[column sep=2cm]
Par(F^{\mathrm{str}}(S)) \rar[dashed]{Par(F^{\mathrm{str}}(f))} \dar[swap]{\beta_{S}} & Par(F^{\mathrm{str}}(S')) \dar{\beta_{S'}}  \\
F(Par(S)) \rar{F(Par f)} & F(Par(S'))
\end{tikzcd}
\end{equation}
where $\beta$ was defined in the proof of \cref{thm:55}. 
Now, let $\alpha: F \Longrightarrow G$ be a monoidal natural transformation of functors $\C \to  \D.$ Define $\mathrm{str}(\alpha)= \alpha^{\mathrm{str}}$ as the unique monoidal natural transformation making the diagram
$$  
\begin{tikzcd}[row sep=1.6cm, column sep=1.6cm]
\Cstr \arrow[swap,bend right=10,""{name=DD}]{rd}{G^{\mathrm{str}}}  \arrow[bend left=25,""{name=UU}]{rd}{F^{\mathrm{str}}} & \arrow[Rightarrow, from=UU,to=DD,shorten=2pt,"\alpha^{\mathrm{str}}"] \\
\C \arrow[hook]{u}{i} \rar[bend right = 20, swap, ""{name=D}]{iG} \rar[bend left = 20, ""{name=U}]{iF} & \D^{\mathrm{str}} \arrow[Rightarrow, from=U,to=D,shorten=2pt,"i\alpha"]
\end{tikzcd}
$$
commute according to \cref{prop:str_nat_trans}. Unravelling definitions, this means that
 $\alpha^{\mathrm{str}}$ is given by $$\alpha^{\mathrm{str}}_S = i(\alpha_{X_1}) * \cdots * i(\alpha_{X_n}):  F^{\mathrm{str}}(S) \to G^{\mathrm{str}} (S)$$ or, equivalently, by the condition that $$ Par(\alpha^{\mathrm{str}}_S) = ( \cdots ( \alpha_{X_1} \boxtimes \alpha_{X_2}) \boxtimes \cdots ) \boxtimes \alpha_{X_n}  : Par(F^{\mathrm{str}}(S)) \to Par(G^{\mathrm{str}}(S)) $$
where $S=(X_1, \ldots , X_n)$. It is clear from these descriptions that for the vertical composition we have $(\alpha_2 \circ  \alpha_1)^{\mathrm{str}}= \alpha_2^{\mathrm{str}} \circ \alpha_1^{\mathrm{str}} $ and that the identity natural transformation maps to the identity natural transformation, so that for any pair of monoidal categories $\C, \D$,   strictification induces a functor $$\mathrm{str}_{\C, \D} : \mathsf{MonCat} (\C, \D) \to \mathsf{strMonCat} (\Cstr, \D^{\mathrm{str}}).$$

\begin{proposition}\label{prop:222}
Strictification defines a 2-functor
$$\mathrm{str}: \mathsf{MonCat} \to \mathsf{strMonCat}.$$
\end{proposition}
\begin{proof}
It is only left to check the equalities \eqref{eq:Anotherdiagram1} and \eqref{eq:Anotherdiagram2}. The first equality \eqref{eq:Anotherdiagram1} amounts to the equality of functors
\begin{equation}\label{eq:000}
(F_2 \circ F_1)^{\mathrm{str}} =F_2^{\mathrm{str}} \circ F_1^{\mathrm{str}}
\end{equation}
for $F_1: \C \to \D$ and $F_2: \D \to \mathcal{E}$, and the equality of natural transformations
\begin{equation}\label{eq:001}
(\alpha_2 * \alpha_1 )^{\mathrm{str}} = \alpha_2^{\mathrm{str}} * \alpha_1^{\mathrm{str}},
\end{equation}
where if $\alpha_i: F_i \Longrightarrow G_i$ then $\alpha_2 * \alpha_1: F_2 F_1 \Longrightarrow G_2G_1$ denotes the horizontal composition. Now, \eqref{eq:000} holds trivially on objects, and on arrows we consider the following commutative diagram:
\begin{equation*}
\begin{tikzcd}[column sep=1.5cm]
Par(F_2^{\mathrm{str}}F_1^{\mathrm{str}}S) \dar[swap]{\beta^2_{F_1^{\mathrm{str}}S}} \rar{F_2^{\mathrm{str}}F_1^{\mathrm{str}}(f)} & Par(F_2^{\mathrm{str}}F_1^{\mathrm{str}}S') \dar {\beta^2_{F_1^{\mathrm{str}}S'}} \\
F_2(Par(F_1^{\mathrm{str}}S)) \rar{F_2(F_1^{\mathrm{str}}(f))} \dar[swap]{F_2\beta^1_{S}} & F_2(Par(F^{\mathrm{str}}S')) \dar{F_2\beta^1_{S'}}  \\
F_2F_1(Par(S)) \rar{F_2F_1(Par f)} & F_2F_1(Par(S'))
\end{tikzcd}
\end{equation*}
In this diagram, we have put $\beta^i=\beta^{F_i}$. The bottom square is the image under $F_2$ of the square \eqref{eq:def66}  for $F_1^{\mathrm{str}}$, and the top square is the square \eqref{eq:def66} for $F_1^{\mathrm{str}}$  and the arrow $F_1^{\mathrm{str}}(f)$. Now the key point to note is that  $F_2\beta_S^1 \circ \beta^2_{F_1^{\mathrm{str}}S} =\beta_S^{F_2 F_1}$, which readily follows from the definition of the monoidal constraint of the composite of monoidal functors. Moreover, \eqref{eq:001} follows from the following computation (where we have removed the paranthesisation of the strictification of a monoidal natural transformation for clarity), where we take $S=(X_1, \ldots, X_n) \in \Cstr$:
\begin{align*}
(\alpha_2 * \alpha_1 )_{X_1} \odot \cdots &\odot (\alpha_2 * \alpha_1 )_{X_n} = \\ &= [G_2(\alpha_{1,X_1}) \circ \alpha_{2, F_1X_1}] \odot \cdots \odot [G_2(\alpha_{1,X_n}) \circ \alpha_{2, F_1X_n}]\\
&=[G_2(\alpha_{1,X_1}) \odot \cdots \odot G_2(\alpha_{1,X_n})] \circ [\alpha_{2, F_1X_1} \odot \cdots \odot \alpha_{2, F_1X_n}] \\
&=G_2^{^{\mathrm{str}}} ((\alpha_1^{\mathrm{str}})_S) \circ (\alpha_2^{\mathrm{str}})_{F_1^{\mathrm{str}}S}.
\end{align*}
Finally, \eqref{eq:Anotherdiagram2} amounts to the equalities $(\id_{\C})^{\mathrm{str}}=\id_{\Cstr}$ and $(\id_{\id_\C})^{\mathrm{str}}=\id_{\id_{\Cstr}}$, where by $\id_{\id_\C}$ we denote the identity natural transformation of the identity functor $\id_\C$. These hold trivially by the definitions.
\end{proof}

Observe that there is a canonical forgetful 2-functor 
$$U: \mathsf{strMonCat} \to \mathsf{MonCat}$$
sending every strict monoidal category, strict monoidal functor and monoidal natural transformation to their underlying monoidal category, underlying strong monoidal functor and the same natural transformation.

The following result upgrades \cref{thm:55} to the 2-categorical level, stating that the strictification is part of a free-forgetful adjunction at the 2-categorical level:

\begin{theorem}\label{thm:2adj_str}
There is a 2-adjunction
$$\begin{tikzcd}[column sep={4em,between origins}]
\mathsf{MonCat}
  \arrow[rr, bend left, swap, "{\mathrm{str}}"' pos=0.50]
& \bot &  \mathsf{strMonCat}
  \arrow[ll, bend left, swap, "U"' pos=0.50]
\end{tikzcd}$$
\end{theorem}
\begin{proof}
For every monoidal category $\C$ and every strict monoidal category $\D$, \cref{thm:55} and \cref{prop:str_nat_trans} together give an isomorphism of categories $$\varphi_{\C, \D} : \mathsf{strMonCat}(\Cstr, \D) \toiso \mathsf{MonCat}(\C,U \D),$$
so it is only left to check that this isomorphism is natural on $\C$ and $\D$.

Let $H: \C_2 \to \C_1$ be a strong monoidal functor between monoidal categories, and let $F:  \C_1  \to \D$ be a strong monoidal functor.  Now consider the following diagram of categories and functors:
$$
\begin{tikzcd}
\Cstr_2 \rar{H^{\mathrm{str}}} & \Cstr_1 \arrow{dr}{\widehat{F}} & \\
\C_2 \uar[hook]{i_2} \rar{H} & \C_1 \uar[hook]{i_1} \rar{F} & \D 
\end{tikzcd}
$$
The leftmost square commutes by the definition of $H^{\mathrm{str}}$, so the whole diagram commutes. This means that $F\circ H = \widehat{F} \circ H^{\mathrm{str}} \circ i_2$, which is precisely the naturality on $ \C$ for functors. The naturality for natural transformations follows by a  similar argument from the following:  if $H,H': \C_1 \to \D$ are strong monoidal functors and $\varepsilon: H \Longrightarrow H'$ is a monoidal natural transformation, then we have the equality of natural transformations $$\varepsilon^{\mathrm{str}} i_2 = i_1 \varepsilon ,$$ which readily follows from the definition of $\varepsilon ^{\mathrm{str}}$.

The naturality on $\D$ is a tautology since $U$ simply forgets the strict structure.
\end{proof}

\subsection{Non-strictification as a 2-functor}

Now let us present a similar construction for the non-strictification, and showing that it similarly gives rise to a 2-functor
$$q: \mathsf{MonCat} \to \mathsf{nonstrMonCat}$$
which will be shown to be left-adjoint to the canonical forgetful functor $$U: \mathsf{nonstrMonCat} \to \mathsf{MonCat}.$$

For a monoidal category $\C$, set  $q(\C) := \C_q$. If $F: \C \to \D$ is a strong monoidal functor between monoidal categories, we define $q(F)=F_q$   as the unique strict monoidal functor $F_q: \C_q \to \D_q$ making the diagram 
$$  
\begin{tikzcd}
\C_q \arrow[dashed]{rrd}{F_q} & & \\
\C \arrow[hook]{u}{i} \rar{F} & \D \arrow[hook]{r}{j} & \D_q
\end{tikzcd}
$$
commutative according to \cref{thm:unique_nonstr}. More explicitly, on objects we have $$ F_q((X_1, \ldots , X_n),t) = ((FX_1,  \ldots , FX_n),t)  $$ and $F_q(\emptyset, \emptyset) = (\emptyset, \emptyset)$. For an arrow $f: (S,t) \to (S',t')$ in $\C_q$, we let $F_q(f)$ be determined by following dashed arrow:
\begin{equation}\label{eq:def661}
\begin{tikzcd}[column sep= 2cm]
Par(F_q(S,t)) \rar[dashed]{Par(F_q(f))} \dar[swap]{\beta_{(S,t)}} & Par(F_q(S',t')) \dar{\beta_{(S',t')}}  \\
F(Par(S,t)) \rar{F(Par f)} & F(Par(S',t'))
\end{tikzcd}
\end{equation}
where $\beta$ is this time as defined in the proof of \cref{thm:unique_nonstr}. 
Let us now describe the non-strictification for 2-cells. Let $F,G: \C \to \D$ be strong monoidal functors between monoidal categories and let $\alpha: F \Longrightarrow G$ be a monoidal natural transformation. Define $q(\alpha)= \alpha_q$ as the unique monoidal natural transformation $\alpha_q: F_q \Longrightarrow G_q$  making the diagram
$$\begin{tikzcd}[row sep=1.6cm, column sep=1.6cm]
\C_q \arrow[swap,bend right=10,""{name=DD}]{rd}{G_q}  \arrow[bend left=25,""{name=UU}]{rd}{F_q} & \arrow[Rightarrow, from=UU,to=DD,shorten=2pt,"\alpha_q"] \\
\C \arrow[hook]{u}{j} \rar[bend right = 20, swap, ""{name=D}]{jG} \rar[bend left = 20, ""{name=U}]{jF} & \D_q \arrow[Rightarrow, from=U,to=D,shorten=2pt,"j\alpha"]
\end{tikzcd}
$$
commute according to \cref{prop:nonstr_nat_trans}. It is readily seen that $(\alpha_q)_{(S,t)}: F_q(S,t) \to G_q(S,t) $ is determined objectwise by the condition that
$$ Par((\alpha_q)_{(S,t)}):= \alpha_{X_1} \boxtimes \cdots \boxtimes \alpha_{X_n} : Par(F_q(S,t)) \to Par(G_q(S,t)),   $$ where $S=(X_1, \ldots , X_n)$ and $\alpha_{X_1} \boxtimes \cdots \boxtimes \alpha_{X_n}$ is supposed to be parenthesised according to $t \in \mathrm{Mag}(\bullet)$. It is easy to see that $(\alpha_2 \circ \alpha_1)_q = (\alpha_2)_q \circ (\alpha_1)_q$. The upshot of this discussion is that we obtain a functor $$q_{\C,\D} :  \mathsf{MonCat} (\C, \D) \to \mathsf{nonstrMonCat} (\C_q, \D_q). $$

\begin{proposition}\label{prop:877}
Non-strictification defines a 2-functor
$$q: \mathsf{MonCat} \to \mathsf{nonstrMonCat}.$$
\end{proposition}
\begin{proof}
The argument is entirely analogous to the one given in \cref{prop:222}.
\end{proof}

\begin{theorem}\label{thm:7}
There is a 2-adjunction
$$\begin{tikzcd}[column sep={4em,between origins}]
\mathsf{MonCat}
  \arrow[rr, bend left, swap, "{q}"' pos=0.50]
& \bot & \phantom{aaaa} \mathsf{nonstrMonCat}
  \arrow[ll, bend left, swap, "U"' pos=0.50]
\end{tikzcd}$$
\end{theorem}
\begin{proof}
If $\C$ is a monoidal category and $\D$ is a non-strict monoidal category, then \cref{thm:unique_nonstr} and \cref{prop:nonstr_nat_trans} define an isomorphism of categories
$$\varphi_{\C, \D} : \mathsf{nonstrMonCat}(\C_q, \D) \toiso \mathsf{MonCat}(\C,U \D).$$ The proof that this isomorphism is natural on $\C$ and $\D$ is identical to that given in \cref{thm:2adj_str}.
\end{proof}

\subsection{Relation between the strictifcation and non-strictification}

We have seen that any monoidal category is monoidally equivalent to both a strict and a non-strict monoidal category. This means that, up to monoidal equivalence, there is no distinction between strict and non-strict categories. We would like to make this idea more precise.

Let $\mathsf{C}$ be a 2-category. The \textit{core truncation} of $\mathsf{C}$ is the ordinary category $h \mathsf{C}$ whose objects are those of $\mathsf{C}$ and whose arrows are 2-isomorphism classes of 1-cells \cite{banks}. More precisely, given a (small) ordinary category $\mathcal{C}$, let  $[\mathcal{C}]$ be the discrete category whose objects are isomorphism classes of objects of $\mathcal{C}$. Trivially this defines a functor $[ - ]: \mathbf{Cat} \to \mathbf{Cat}$. Then the set of arrows in $h \mathsf{C}$ is defined as $$\hom{h \mathsf{C}}{A}{B} := \hom{\mathbf{Cat}}{*}{[\mathsf{C}(A,B)]},$$ where we recall that $*$ denotes the one-object discrete category. The identities and the composite law are inherited from those of $\mathsf{C}$ in the obvious way. Note that the core truncation satisfies the property that if $f \in \mathsf{C}(A,B)$ is an equivalence (cf. \cref{subsec:2-categories}), then the arrow $A \to B$ corresponding to the 2-isomorphism class of $f$ is an isomorphism. 

If $F: \mathsf{C} \to \mathsf{C}'$ is a 2-functor between 2-categories, then define a functor $$ hF : h \mathsf{C} \to h\mathsf{C}'$$ as  $(hF)(A) := F(A)$ on objects and 
$$
\begin{tikzcd}
 \hom{\mathbf{Cat}}{\bm{1}}{[\mathsf{C}(A,B)]} \rar{[F_{A,B}]_*} & \hom{\mathbf{Cat}}{\bm{1}}{[\mathsf{C}'(FA,FB)]} 
\end{tikzcd}
$$
on arrows. This defines the \textit{core truncation functor}
$$ h:  2 {\text -}\mathbf{Cat} \to \mathbf{Cat} .$$
 
\begin{remark}
For the model/homotopical category theory minded reader, we would like to place the core truncation construction into a more general context. Suppose that $\C$ is a closed symmetric monoidal homotopical category \cite{riehl2014}, and $\V$ is a $\C$-enriched category. If $\mathrm{Ho}(\C)$ denotes the homotopy category of $\C$, then $\V$ gives rise to a $\mathrm{Ho}(\C)$-enriched category $\bar{h} \V$, where the hom-objects are given by the images of the hom-objects $\V(A,B)$ under the localisation functor $\C \to \mathrm{Ho}(\C)$. These hom-objects $\bar{h} \V(A,B)$ are thought of as ``homotopy classes''. We can furthermore take the underlying ordinary category of $\bar{h} \V$, that we denote $h \V := (\bar{h} \V)_0$. Note that we have
$$\hom{h \V}{A}{B} = \hom{\mathrm{Ho}(\C)}{\bm{1}}{\bar{h} \V(A,B)} = \hom{\mathrm{Ho}(\C)}{\bm{1}}{\V(A,B)}.$$
Having said that, take $\C= \mathbf{Cat}$, the category of small categories with its canonical symmetric monoidal model (hence homotopical) structure, where the weak equivalences are given by the equivalences of categories. The hom-sets in $\mathrm{Ho}(\mathbf{Cat})$ are given by natural isomorphism classes of functors. If $\mathsf{C}$ is a 2-category (i.e. a $\mathsf{Cat}$-enriched category), then
$$\hom{\mathrm{Ho}(\mathbf{Cat})}{*}{\mathsf{C}(A,B)} = \hom{\mathbf{Cat}}{*}{[\mathsf{C}(A,B)]},$$
because the datum of a natural isomorphism between functors $i_f, i_g : * \to  \mathsf{C}(A,B)$ is equivalent to the datum of a 2-isomorphism $\alpha: f \to g$. This makes our notation $h \mathsf{C}$ consistent, and realises the core truncation as an instance of a construction in homotopical category theory.
\end{remark}

\begin{lemma}
The core truncation functor $ h:  2 {\text -}\mathbf{Cat} \to \mathbf{Cat} $  preserves adjunctions, that is, if 
$$\begin{tikzcd}[column sep={4em,between origins}]
\V
  \arrow[rr, bend left, swap, "{F}"' pos=0.50]
& \bot &  \W
  \arrow[ll, bend left, swap, "G"' pos=0.50]
\end{tikzcd}$$
is a 2-adjunction, then 
$$\begin{tikzcd}[column sep={4em,between origins}]
h\V
  \arrow[rr, bend left, swap, "{hF}"' pos=0.50]
& \bot & h \W
  \arrow[ll, bend left, swap, "hG"' pos=0.50]
\end{tikzcd}$$
is an ordinary adjunction.
\end{lemma}
\begin{proof}
This follows directly applying the functor $$\hom{\mathbf{Cat} }{*}{[-]}: \mathbf{Cat}  \to \mathbf{Set}$$ to  each of the isomorphisms of categories  $$\varphi_{V,W}: \W(FV, W) \toiso \V(V, GW)$$ of the 2-adjunction and to each of the commutative diagrams expressing the naturality.
\end{proof}

\begin{theorem}
The core truncations of the (non-)strictification 2-functors define a pair of adjoint equivalences (i.e., adjunctions and equivalences of categories)
$$\begin{tikzcd}[column sep={4em,between origins}]
h\mathsf{nonstrMonCat}    \arrow[rr, bend right, "hU"' pos=0.50]
& \bot &  h\mathsf{MonCat} \arrow[ll, bend right,  "{hq}"' pos=0.50]
  \arrow[rr, bend left, swap, "{h \mathrm{str}}"' pos=0.50] \arrow[rr, bend left, swap, ""' pos=0.50] 
& \bot &  h\mathsf{strMonCat}
  \arrow[ll, bend left, swap, "hU"' pos=0.50]
\end{tikzcd}$$
\end{theorem}
\begin{proof}
That both pairs are adjunctions follows from the previous lemma. To see that they are also equivalences of categories, recall that given an adjunction $ F \dashv G$, we have that $F$ is fully faithful if and only if the unit of the adjunction is a natural isomorphism (e.g. \cite[4.5.13]{riehl2017}). The unit $\eta$ of the adjunction  $h \mathrm{str} \dashv hU$ is objectwise precisely the 2-isomorphism class of the functor $i: \C \to \Cstr$, hence $\eta$ must be a natural isomorphism in light of \cref{thm:1}, so that $h \mathrm{str}$ is fully faithful.  On the other hand, if $\D$ is a strict monoidal category, then we have that $Par: \D^{ \mathrm{str}} \to \D$ is a strict equivalence of categories, which means that $\D$ and $ \D^{ \mathrm{str}} =h \mathrm{str}(\D)$ are isomorphic in $ h\mathsf{strMonCat}$, and therefore $ h  \mathrm{str} $ is essentially surjective. For the non-strictification, the argument is identical.
\end{proof}

\section{Reconstruction via 2-monads}\label{sec:5}

In the previous sections we have given concrete models for the (non-)strictification of monoidal categories based on (parenthesised) sequences of objects. In this section, we want to show that these two constructions, as well as the corresponding monoidal equivalences with the original monoidal category, can be regarded as instances of Power's coherence result for 2-monads, when applied to the free monoid (resp. magma) 2-monad.

\subsection{Monads and their algebras} We start by briefly recalling some basics about monads and its algebras in ordinary categories. We refer the reader to \cite{maclane,riehl2017} for a detailed exposition.

If  $\C$ is a category, a \textit{monad} is a monoid object in the monoidal category $\mathsf{Cat}(\C , \C)$ of endofunctors of $\C$. More precisely, a monad is a triple $(T, \mu, \eta)$ where $T: \C \to \C$ is an endofunctor and $\mu: T^2 \Longrightarrow T$ and $\eta: \id_\C \Longrightarrow T$ are natural transformations satisfying the following associativity and unitality conditions:  $$  \mu \circ T\mu = \mu \circ \mu T  \qquad, \qquad \mu \circ T\eta = \id_T = \mu \circ \eta T. $$
An \textit{algebra} over $(T, \mu, \eta)$, or simply a \textit{$T$-algebra}, is a pair $(A , a)$ formed by an object $A$ in $\C$ and an arrow $a: T(A) \to A$ satisfying 
\begin{equation}\label{eq:T-algebra}
a \circ Ta   =a \circ \mu_A   \qquad , \qquad \id_A=  a\circ \eta_A.
\end{equation}
A \textit{$T$-algebra homomorphism} $f: (A,a) \to (B,b)$ between $T$-algebras is a map $f: A \to B$ in $\C$ compatible with the structure maps $a,b$ in the sense that $$Tf \circ b= f \circ a.$$ $T$-algebras and $T$-algebra homomorphisms form a category $T\text{-}\mathbf{Alg}$, called the category of $T$-algebras or the \textit{Eilenberg-Moore category} of $T$.

Monads are very closely related to adjunctions. If
$$\begin{tikzcd}[column sep={4em,between origins}]
\C
  \arrow[rr, bend left, swap, "{F}"' pos=0.50]
& \bot &  \D
  \arrow[ll, bend left, swap, "U"' pos=0.50]
\end{tikzcd}$$
is an adjunction with unit $\eta: \id_\C \Longrightarrow UF $ and counit $\eta: FU \Longrightarrow \id_\D$, then it induces a monad $(T, \mu, \eta)$ given by $$ T:= UF \quad , \quad \mu:= U \varepsilon F : UFUF =T^2 \Longrightarrow T= UF \quad , \quad \eta: \id_\C  \Longrightarrow T=UF . $$
Conversely, any monad $(T, \mu, \eta)$ gives rise to an adjunction between $\C$ and the category of $T$-algebras, 
\begin{equation}\label{eq:FT-UT}
\begin{tikzcd}[column sep={4em,between origins}]
\C
  \arrow[rr, bend left, swap, "{F^T}"' pos=0.59]
& \bot &  T\text{-}\mathbf{Alg}
  \arrow[ll, bend left, swap, "U^T"' pos=0.40].
\end{tikzcd}
\end{equation}
Here, for an object $X$ in $\C$, $F^T(X)$ is the \textit{free $T$-algebra} $(TX, \mu_X)$, whereas for a $T$-algebra $(A,a)$, we have $U^T(A,a):=A$ so $U^T$ is simply the forgetful functor.

Every monad $(T,\mu, \eta)$ arises from an adjunction, namely from the free $T$-algebra adjunction  \eqref{eq:FT-UT} above. However, not every adjunction comes from a monad, only \textit{monadic adjunctions} do, that is, those adjunctions $F\dashv U$ such that the functor $K:\D \to T\text{-}\mathbf{Alg}$, $K(Y):= (UY,U\varepsilon_Y )$ is an equivalence of categories.

\begin{example}\label{ex:monoid_monad}
Let $T: \mathbf{Set} \to \mathbf{Set}$ be the \textit{free monoid monad}, $T(X):= \mathrm{Mon}(X)$ the underlying set of the free  monoid on $X$. As sets, $$T(X)= \coprod_{n\geq 0 } X^n.$$ The multiplication $\mu_X : \mathrm{Mon}(\mathrm{Mon}(X)) \to \mathrm{Mon}(X)$ is the map that realises a word in the alphabet $\mathrm{Mon}(X)$ as a word in the alphabet $X$, and the unit $\eta_X: X \to \mathrm{Mon}(X)$ assigns to every element of $X$ the single-letter word. Obviously, this monad is induced by the free monoid adjunction.

If $(X, a: \mathrm{Mon}(X) \to X)$ is a $T$-algebra, let us write $a_n: X^n \to X$ for the restriction to $X^n$ map. Consider an element $$(x_{11} \cdots x_{1k_1})\cdots (x_{n1} \cdots x_{nk_n})\in  \mathrm{Mon}(\mathrm{Mon}(X)).  $$   The left-hand side equation of \eqref{eq:T-algebra} translates into
\begin{equation}\label{eq:strict_associativity}
a_k(a_{k_1}(x_{11} \cdots x_{1k_1}) \cdots a_{k_n}(x_{n1} \cdots x_{nk_n}))=a_n(x_{11} \cdots  x_{nk_n}) 
\end{equation}
with $k= \sum_i k_i$ and $x_{ij}\in X$. In particular, given elements $x_1, x_2, x_3 \in  X$, the previous equation says (using it twice with $(x_1 x_2)x_3$ and $x_1(x_2x_3)$) that
\begin{equation}\label{eq:associativity_monoid_monad}
a_2(a_2(x_1x_2)x_3)= a_3(x_1x_2x_3)=a_2(x_1a_2(x_2x_3)) , 
\end{equation}
which says that $a_2: X \times X \to X$  is associative and, by an inductive argument, that $a_n$ for $n>2$ can be expressed in terms of $a_2$. On the other hand, it is readily seen that the right-hand side of \eqref{eq:T-algebra} translates into $a_1=\id$. Now, let $1:= a_0(\emptyset)$.  The left-hand side of \eqref{eq:T-algebra} also implies (using it twice) that
\begin{equation}\label{eq:unitality_monoid_monad}
a_2(1 x)=x = a_2(x1).
\end{equation}
The upshot is that $a$  amounts to the map $a_2: X \times X \to X$ together with a unit $1 \in X$  with the property that $a_2$ is associative and unital, hence the data of a $T$-algebra $(X,a)$ is equivalent to the data of a monoid structure on $X$. It is easy to see that a $T$-algebra map  amounts to a giving a monoid map. Therefore, the category $T\text{-}\mathbf{Alg}$ is equivalent to the category of monoids.
\end{example}

\begin{example}\label{ex:free_magma_mon}
Let $T: \mathbf{Set} \to \mathbf{Set}$ be the \textit{free magma monad}, $T(X):= \mathrm{Mag}(X)$ the underlying set of the free (unital) magma on $X$. As sets, $$T(X)= \coprod_{n\geq 0 } X_n \quad \mathrm{where} \quad X_0:= \{ \emptyset \}, \ X_1:= X, \ X_{n} := \coprod_{p+q=n} X_p \times X_q, \ (n>1).$$ Similarly to the monoid case, the multiplication $\mu_X : \mathrm{Mag}(\mathrm{Mag}(X)) \to \mathrm{Mag}(X)$ is the map that realises a parenthesised word in the alphabet $\mathrm{Mag}(X)$ as a parenthesised word in the alphabet $X$, and the unit is exactly as in the monoid case. Obviously, this monad is induced by the free magma adjunction.

An argument similar to the one of the monoid case allows to conclude that the category $T\text{-}\mathbf{Alg}$ is equivalent to the category of magmas.

Note that $T(X)= \mathrm{Mag}(X)$ is naturally bijective with the set
\begin{equation}\label{eq:free_magma_monoid_cool}
T(X)= \coprod_{n \geq 0} X^{\times n} \times \mathrm{Mag}_n(\bullet),
\end{equation}
where $\mathrm{Mag}_n(\bullet) := |-|^{-1} (n)$ (cf. \cref{subseq:non-strict}).
\end{example}

\begin{remark}
The reader should note that, in both previous examples, the category $\mathbf{Set}$ could be replaced for an arbitrary monoidal category $\C$ with finite coproducts in which the monoidal product preserves them in each variable. In such cases the Eilenberg-Moore categories are equivalent to the categories of monoids and magma objects in $\C$, respectively.
\end{remark}

\subsection{2-monads and their algebras}

We now want to focus on the 2-categorical counterparts of monads, which will give us more flexibility. This is the relevant definition: if $\CC$ is a 2-category,  a \textit{2-monad} in $\CC$ is  a triple $(T, \mu, \eta)$ where $T:\CC \to \CC$ is a endo-2-functor and $\mu: T^2 \Longrightarrow T$ and $\eta: \id_\CC \Longrightarrow T$ are 2-natural transformations satisfying $$  \mu \circ T\mu = \mu \circ \mu T  \qquad, \qquad \mu \circ T\eta = \id_T = \mu \circ \eta T, $$ just as in the monad case.

For the algebras, we can take advantage of the higher categorical context and allow more flexibility, by requiring not equalities  as in  \eqref{eq:T-algebra} but simply that those 1-cells are related by some coherent 2-cells. More precisely, a \textit{pseudo $T$-algebra} is a quadruple $(A , a, \theta, \nu)$ formed by an object $A$ in $\CC$, a 1-cell $a: T(A) \to A$ and a couple of invertible 2-cells 
\begin{equation}\label{eq:axioms_2monad_alg}
\theta: a \circ Ta  \toiso a \circ \mu_A   \qquad , \qquad \nu: \id_A \toiso a\circ \eta_A ,
\end{equation}
which satisfy the following coherence conditions expressing the (higher) associativity and unitality of $\theta$ and $\nu$,
\begin{equation}\label{eq:axioms_2monad_alg_1}
\begin{tikzcd}
a \circ \id_{TA} \arrow[equal]{d} \rar{\id_a * T \nu_A} & a \circ Ta \circ T \eta_A \dar{\theta * \id_{T\eta_A}} \\
a \arrow[equal]{r} & a \circ \mu_A \circ T\eta
\end{tikzcd} \quad , \quad
\begin{tikzcd}
\id_a \circ a \rar{\nu * \id_a} \arrow[equal]{d} & a \circ \eta_A \circ a  \arrow[equal]{r} & a \circ Ta \circ \eta_{TA} \dar{\theta * \id_{\eta_{TA}}} \\
a \arrow[equal]{rr} & & a \circ \mu_A \circ \eta_{TA}
\end{tikzcd}
\end{equation}
\begin{equation}\label{eq:axioms_2monad_alg_2}
\begin{tikzcd}
a \circ Ta \circ T^2a \dar{\id_a * T\theta} \rar{\theta * \id_{T^2a}} & a \circ \mu_A \circ T^2a \arrow[equal]{r} & a \circ Ta \circ \mu_{TA} \dar{\theta * \id_{\mu_{TA}}}\\
a \circ Ta \circ T\mu_A \rar{\theta * \id_{T\mu_A}} & a \circ \mu_A \circ T \mu_A \rar[equal] & a \circ \mu_A \circ \mu_{TA}
\end{tikzcd}
\end{equation}
 If $\theta$ and $\nu$ are the identity 2-cells, then we say that $(A , a)$ is a \textit{strict $T$-algebra}.

Given pseudo-$T$-algebras $(A, a, \theta, \nu)$ and  $(B, b, \theta', \nu')$, a \textit{pseudomorphism} is a pair $(f, \delta)$ where $f: A \to B$ is a 1-cell and $$\delta: b \circ Tf \toiso f \circ a   $$ is an invertible 2-cell which satisfies certain coherence axioms which express compatibility with the rest of structure 2-cells of $A$ and $A'$, see  \cite[(2.4)--(2.5)]{power}. If $\delta$ is the identity 2-cell we say that $f$ is a \textit{strict morphism}. We will denote by $\mathbf{pseudo}\textrm{-}T\textrm{-}\mathbf{Alg}$ the category of pseudo-$T$-algebras and pseudomorphisms, and by   $\mathbf{str}\textrm{-}T\textrm{-}\mathbf{Alg}$ the category of strict $T$-algebras and strict morphisms

\subsection{Strictification via 2-monads}\label{subsec:strict_via_2monads}
Let $\mathsf{Cat}$ be the 2-category of categories, functors and natural transformations. Inspired by \cref{ex:monoid_monad}, we define the \textit{free monoid 2-monad}  $T:\mathsf{Cat} \to \mathsf{Cat}$ as the following 2-functor: on 0-cells, $$T(\C) := \coprod_{n \geq 0} \C^{\times n},$$ where $\C^{\times 0}=*$ is the one-object discrete category. More explicitly, $T(\C)$ is the category whose objects are finite sequences $(X_1, \ldots , X_n)$ of objects of $\C$, and whose arrows are  $$\hom{T(\C)}{(X_1, \ldots , X_n)}{(Y_1, \ldots , Y_n)} = \prod_{i=1}^n \hom{\C}{X_i}{Y_i} $$ and no arrows between sequences of different length, with composition given componentwise. A functor $F: \C \to \D$ is sent to $T(F):= \coprod_{n \geq 0} F^{\times n}$ and a natural transformation $\alpha: F \Longrightarrow G$ to $T(\alpha):= \coprod_{n \geq 0} \alpha^{\times n}$. Similarly to \cref{ex:monoid_monad}, $\mu_\C: T^2(\C) \to T(\C)$ is given by regarding a sequence of sequences of objects of $\C$ as a sequence of objects of $\C$ (removing inner parenthesis), and $\eta_\C: \C \to T(\C)$ is the canonical inclusion functor.

Let us now suppose we are given a strict $T$-algebra $(\C, a:T(\C)\to \C)$. As in \cref{ex:monoid_monad}, $a$ is the disjoint union of individual functors $a_n: \C^{\times n} \to \C$. In this case \eqref{eq:axioms_2monad_alg} (which are  equalities in the strict case) implies, by the same argument used in \eqref{eq:associativity_monoid_monad} and \eqref{eq:unitality_monoid_monad}, that $\C$ is a strict monoidal category with monoidal product $\otimes := a_2$ and unit $\bm{1} := a_0(*)$.

\begin{remark}\label{rem:unbiased}
In fact, the data of a strict $T$-algebra $(\C, a:T(\C)\to \C)$ is precisely the same  data as what we call an \textit{unbiased strict monoidal category}, since the standard definition from \cref{subsec:mon_cats} is understood to be biased towards arities 2 and 0.
\end{remark}

Let us now consider a pseudo-$T$-algebra $(\C, a, \theta, \nu)$. This amounts to a family of functors $a_n: \C^{\times n} \to \C$ together with natural isomorphisms (compare with \eqref{eq:strict_associativity} and \eqref{eq:associativity_monoid_monad})
$$ \theta_{S}: a_n(a_{k_1}(X_{11}, \ldots , X_{1k_1}), \ldots , a_{k_n}(X_{n1} ,\ldots , X_{nk_n})) \toiso   a_k(X_{11}, \ldots , X_{nk_n})  $$ with $k=\sum k_i$ and $S=((X_{11} ,\ldots , X_{1k_1}), \ldots , (X_{n1} , \ldots , X_{nk_n}))$ and $$ \nu_X : X \toiso a_1(X)  $$  for objects $X, X_{ij}$ in $\C$. In this case the axioms \eqref{eq:axioms_2monad_alg_1} and \eqref{eq:axioms_2monad_alg_2} express that a pseudo-$T$-algebra $(\C, a, \theta, \nu)$ is exactly the same thing as what it is known in the literature as an \textit{unbiased monoidal category} \cite{leinster_hohc}. As we pointed out in \cref{rem:unbiased}, it is unbiased in the sense that there is no preference to the arities 2 and 0, but instead we have a coherent family of functors $a_n: \C^{\times n} \to \C$. Pseudomorphisms and strict morphisms are in this context called \textit{strong} and \textit{strict monoidal functors} (of unbiased monoidal categories), respectively.

There is a canonical functor
\begin{equation}\label{eq:UMoncat_MonCat}
\mathbf{UMonCat}  \to \mathbf{MonCat}
\end{equation}
from the category $\mathbf{pseudo}\textrm{-}T\textrm{-}\mathbf{Alg}=\mathbf{UMonCat}$ of unbiased monoidal categories and strong monoidal functors to the category $\mathbf{MonCat}$ of monoidal categories and strong monoidal functors that assigns, to any unbiased monoidal category $(\C, a, \theta, \nu)$, the following monoidal structure on $\C$: the monoidal product is given by $a_2: \C \times \C \to \C$. Given objects $X,Y,Z$ in $\C$, the associativity constraint is given by the composite  (compare with \eqref{eq:associativity_monoid_monad})
$$
\begin{tikzcd}[row sep= 0.5em, column sep=4em]
a_2(a_2(X,Y),Z)  \rar{a_2(\id, \nu_Z)}[swap]{\cong}  &  a_2(a_2(X,Y),a_1(Z)) \\
\phantom{a_2(a_2(X,Y),Z) } \rar{\theta_{((X,Y),Z)}}[swap]{\cong } & a_3(X,Y,Z) \phantom{----}\\
\phantom{a_2(a_2(X,Y),Z) } \rar{\theta^{-1}_{(X,(Y,Z))}}[swap]{\cong } &  a_2(a_1(X),a_2(Y,Z))\\
\phantom{a_2(a_2(X,Y),Z) } \rar{a_2(\nu_X^{-1},\id)}[swap]{\cong } &  a_2(X,a_2(Y,Z)) \phantom{--}
\end{tikzcd}
$$
and the left and right unit constraints by
\begin{equation}\label{eq:left_right_constr_pseudo}
\begin{aligned}
&\begin{tikzcd}[row sep= 0.5em, column sep=3em]
a_2(\bm{1},X)  \rar{a_2(\id ,\nu_X)}[swap]{\cong}  &  a_2(a_0(*),a_1(X))  \rar{\theta_{((X),*)}}[swap]{\cong } & a_1(X)  \rar{\nu^{-1}_X}[swap]{\cong } &  X
\end{tikzcd}\\
&\begin{tikzcd}[row sep= 0.5em, column sep=3em]
a_2(X,\bm{1})  \rar{a_2(\nu_X,\id )}[swap]{\cong}  &  a_2(a_1(X),a_0(*))  \rar{\theta_{(*,(X))}}[swap]{\cong } & a_1(X)  \rar{\nu^{-1}_X}[swap]{\cong } &  X
\end{tikzcd}
\end{aligned}
\end{equation}
where $\bm{1} := a_0 (*)$. A strong monoidal functor $(F, \delta): (\C,a, \theta, \nu) \to (\D, b, \theta', \nu')$ between unbiased monoidal categories is sent to the same underlying ordinary functor $F: \C \to \D$ with monoidal constraint $\gamma$ being the restriction of $\delta$, 
\begin{equation}\label{eq:constaint_delta_gamma}
\gamma_{X,Y}:= \delta_{(X,Y)} : b_2(FX,FY)= (b_2 \circ TF )(X,Y) \toiso F(a_2(X,Y))
\end{equation}
and $u:= \delta_* : b_0(*) \toiso F(a_0(*))$.

\begin{theorem}[e.g. \cite{leinster_hohc}]\label{thm:equiv_UMonCat_MonCat}
The functor \eqref{eq:UMoncat_MonCat} is an equivalence of categories, $$ \mathbf{UMonCat}  \simeq \mathbf{MonCat} .$$
\end{theorem}

The functor \eqref{eq:UMoncat_MonCat} does not have a canonical quasi-inverse, yet this will our preferred choice: if $\C$ is a monoidal category, then given a sequence  $S= (X_1, \ldots , X_n) $ of objects, $n\geq 0$, define $a_n: \C^{\times n} \to \C$ as
\begin{equation}\label{eq:a_n=Par}
a_n(S) := Par(S)
\end{equation}
where $Par$ is as in \eqref{eq:Par}. The constraint $\theta$ is defined inductively using the associativity constraint as in \eqref{eq:2} (in fact that is a particular case), and $\nu$ is the identity natural transformation. The upshot is, therefore, that a pseudo-$T$-algebra is essentially the same as a monoidal category.

\begin{remark}\label{rem:iso:UstrMonCat_strMonCat}
Observe that \cref{rem:unbiased} and the discussion  above translate into an equivalence of categories 
$$ \mathbf{str}\textrm{-}T\textrm{-}\mathbf{Alg} =  \mathbf{UstrMonCat}  \simeq \mathbf{strMonCat} .$$
\end{remark}

We now move on to discuss a ``strictification'' result for algebras over 2-monads. The theorem below usually goes under the name of Power's general coherence and continues work of Kelly \cite{kelly1,kelly2}. First we need a

\begin{lemma}[\cite{power}]\label{lem:power}
Consider the following diagram of categories and functors
$$
\begin{tikzcd}
\mathcal{A} \rar{F} \dar{J} & \mathcal{B} \dar{G} \\
\C \rar{K} & \D
\end{tikzcd}
$$ and suppose $$ \alpha: G \circ F \overset{\cong}{\Longrightarrow} K \circ J  $$ is a natural isomorphism. If $F$ is bijective-on-objects, and $K$ is fully faithful, then there exist a  unique functor $H: \mathcal{B} \to \C$ and a unique natural isomorphism $\beta: G \overset{\cong}{\Longrightarrow} K\circ H$ such that $H \circ F =J$ and $\alpha = \beta F$.
\end{lemma}
\begin{proof}[About the proof]
The functor $H$ is necessarily defined by $H(X):= J(F^{-1}(X))$ on objects and on arrows as the composite
$$
\begin{tikzcd}[row sep= 0.5em, column sep=6em]
\hom{\mathcal{B}}{X}{Y}  \rar{G} &  \hom{\mathcal{D}}{GX}{GY} = \hom{\mathcal{D}}{GFF^{-1}X}{GFF^{-1}Y} \\
\phantom{\hom{\mathcal{B}}{X}{Y}  } \rar{\alpha_{F^{-1}Y} \circ (-) \circ \alpha^{-1}_{F^{-1}X}} & \hom{\mathcal{D}}{KJF^{-1}X}{KJF^{-1}Y} \phantom{--------}\\
\phantom{\hom{\mathcal{B}}{X}{Y}} \rar{K^{-1}}[swap]{\cong } &  \hom{\mathcal{C}}{JF^{-1}X}{JF^{-1}Y} =   \hom{\mathcal{C}}{HX}{HY}
\end{tikzcd}
$$
\end{proof}

Before stating the theorem, recall that any functor $F: \C \to \D$ canonically factors as a composite $$   \C \overset{F'}{\to} \mathcal{E} \overset{F''}{\to} \D \qquad , \qquad F=F'' \circ F'$$ where $F'$ is bijective-on-objects and $F''$ is fully faithful. The category $\mathcal{E}$ can be realised, up to isomorphism, as having the same objects as $\C$ and arrows $$ \hom{\mathcal{E}}{X}{Y} := \hom{\D}{FX}{FY}  $$ with composite and unit inherited from $\D$.

\begin{theorem}[\cite{power}]\label{thm:power}
Let $T: \mathsf{Cat}' \to \mathsf{Cat}'$ be a 2-monad on the 2-category $\mathsf{Cat}' $ of categories, functors and natural isomorphisms, and suppose that $T$ preserves bijective-on-object functors, that is, if $F:\C \to \D$ is bijective on objects, so is $T(F): T(\C) \to T(\C)$.

Then every pseudo-$T$-algebra is equivalent, via a pseudomorphism, to a strict $T$-algebra.
\end{theorem}
\begin{proof}[About the proof]
Given  a pseudo-$T$-algebra  $(\C, a, \theta, \nu)$ with canonical factorisation of $a$  $$T(\C) \overset{a'}{\to}  \mathcal{E} \overset{a''}{\to } \C \qquad , \qquad a=a'' \circ a',$$ we first apply \cref{lem:power} to the square
$$
\begin{tikzcd}
T^2\mathcal{C} \rar{Ta'} \dar[swap]{a' \circ \mu_\C} & T\mathcal{E} \dar{a \circ Ta''} \\
\mathcal{E} \rar{a''} & \C
\end{tikzcd}
$$ 
with $\alpha= \theta$, which produces a functor $H: T\mathcal{E} \to \mathcal{E}$ and a natural isomorphism $ \beta : a \circ Ta'' \Longrightarrow a'' \circ H  $ such that $H \circ Ta'= a' \circ \mu_\C$ and $\theta = \beta Ta'$. The rest of the proof consists of verifying that $(\mathcal{E},H)$ is a strict $T$-algebra and $(a'', \beta) : \mathcal{E}  \to \C$ is a pseudomorphism. Then we conclude since $a''$ is an equivalence of categories: it is fully faithful by construction and besides essentially surjective since for all objects $X$ in $\C$, we have $\nu_X:X \toiso a'' (a'(\eta_\C (X)))$.
\end{proof}

Let us apply Power's general coherence theorem to the free monoid 2-monad $T :\mathsf{Cat}' \to \mathsf{Cat}'$ from the beginning of the subsection.    First of all observe that $T$ preserves bijective-on-objects functors: if $F: \C \to \D$ is bijective-on-objects, then so is $T(F)$: indeed   $T(F)^{-1}(Y_1, \ldots, Y_n)= (F^{-1}Y_1, \ldots, F^{-1}Y_n)$. Therefore, we can  use the theorem for this 2-monad.

Consider a monoidal category $\C$, that we view as a  pseudo-$T$-algebra $(\C, a, \theta, \nu)$ via the equivalence of \cref{thm:equiv_UMonCat_MonCat}. According to Power's \cref{thm:power}, $\C$ is equivalent to a strict $T$-algebra $\mathcal{E}$, that we can view as a strict monoidal category according to \cref{rem:iso:UstrMonCat_strMonCat}. Let us describe precisely this strict $T$-algebra $\mathcal{E}$.

As an ordinary category, $\mathcal{E}$ has the same objects as $T(\C)$, so finite sequences  $S=(X_1, \ldots , X_n)$ of objects of $\C$, $n\geq 0$. The morphisms of $\mathcal{E}$ are given by
$$
\hom{\mathcal{E}}{S}{S'} = \hom{\mathcal{C}}{a(S)}{a(S')}= \hom{\mathcal{C}}{Par(S)}{Par(S')}.
$$
where we have used \eqref{eq:a_n=Par} in the second equality. That is, as an ordinary category, $\mathcal{E}$ is exactly the category $\Cstr$ that we constructed in \cref{subsec:strictification}. Let us see now the monoidal structure. This is derived from the structure functor $H: T\mathcal{E} \to \mathcal{E}$ that arises from \cref{lem:power}. Given two objects $S,S'$ in $\mathcal{E}$, we have that $$ H_2(S,S') = (a' \circ \mu_\C) (S,S') = S*S' $$ the concatenation of the sequences.  On arrows, $H_2$ is given by the composite (cf. \cref{lem:power})

\begin{align*}
\hom{T\mathcal{E}}{(S_1,S_1')}{(S_2,S_2')} &\to \hom{\C}{a(a(S_1),a(S_1'))}{a(a(S_2),a(S_2))}\\
&\to \hom{\C}{a(S_1*S_1')}{a(S_2*S_2')}\\
&\to \hom{\C}{S_1*S_1'}{S_2*S_2'}
\end{align*}
where the first arrow is $a \circ Ta''$, the second arrow is $\theta_{(S_2,S_2')} \circ (-) \circ \theta^{-1}_{(S_1,S_1')}$ and the third arrow is simply $(a'')^{-1}$. This means that 
$$ H_2((f,g))= (a'')^{-1}(\theta_{(S_2,S_2')} \circ a(f,g) \circ \theta^{-1}_{(S_1,S_1')}).  $$ Taking into account that we view $\C$ as a monoidal category via \eqref{eq:a_n=Par}, comparing with  \eqref{eq:3} we conclude that $\mathcal{E} $ coincides with $\Cstr$.

Lastly, note that $a_1: \mathcal{E} \to \C$ is essentially $Par$ also as monoidal functors. Indeed the  constraint $\delta$  in this case is given by the natural isomorphism $$\beta : a \circ Ta_1 \overset{\cong}{\Longrightarrow} a_1 \circ H $$ from \eqref{thm:power} which satisfies $\theta = \beta Ta_0$, viewed as the  monoidal constaint for an ordinary monoidal category as explained in  \eqref{eq:constaint_delta_gamma}. Since $Ta_0$ is essentially the identity on objects, the conclusion is that $\beta$ is given by $\theta$, as required.


All in all, we have proved the following

\begin{theorem}\label{thm:recovering_equivalence_Cstr_C}
Let $\C$ be a pseudo-$T$-algebra, and let $\mathcal{E}$ be the strict $T$-algebra obtained from \cref{thm:power}. Then viewing $\mathcal{E}$ as a strict monoidal category via the isomorphism from \cref{rem:iso:UstrMonCat_strMonCat}, $\mathcal{E}$ is exactly the strictification $\Cstr$ of $\C$ viewed as a monoidal category via \cref{thm:equiv_UMonCat_MonCat}.

Furthermore, the pseudofunctor $a'': \mathcal{E} \to \C$ is precisely, up to these identifications, the monoidal equivalence $Par: \Cstr \to \C$ from \cref{subsec:strictification}.
\end{theorem}

\subsection{Non-strictification via 2-monads}
Now, inspired by the previous subsection, \cref{ex:free_magma_mon} and particularly \eqref{eq:free_magma_monoid_cool},  define the \textit{free magma 2-monad}  $T:\mathsf{Cat}' \to \mathsf{Cat}'$ as the  2-functor $$T(\C) := \coprod_{n \geq 0} \C^{\times n} \times \mathrm{Mag}_n(\bullet),$$ where $\mathrm{Mag}_n(\bullet)$ is viewed as an  indiscrete category (in particular a groupoid). More explicitly, $T(\C)$ is the category whose objects are pairs $(S,t)$ where $S$ is a finite sequence $(X_1, \ldots , X_n)$ of objects of $\C$ and $t \in \mathrm{Mag}_n(\bullet)$, and whose arrows are tuples of arrows of $\C$ as in \cref{subsec:strict_via_2monads}.  The structure morphism $\mu_\C: T^2(\C) \to T(\C)$ is given as follows: $$ \mu_\C(((S_1,t_1), \ldots, (S_n,t_n)),t):= (S_1 * \cdots *S_n, t_{t_1, \ldots, t_n})  $$
where $t_{t_1, \ldots, t_n} \in \mathrm{Mag}(\bullet)$ stands for the element that arises from inserting $t_i$ in the $i$-th bullet of $t$. Of course $\eta_\C: \C \to T(\C)$ is the canonical inclusion functor $\eta_\C(X)=((X), \bullet)$.

Let us now consider a pseudo-$T$-algebra $(\C, a, \theta, \nu)$. This amounts to a family of functors $a_n: \C^{\times n} \times \mathrm{Mag}_n(\bullet) \to \C$ together with a family of isomorphisms
$$ \theta_P: a_n((a_{k_1}(S_1,t_1), \ldots , a_{k_n}(S_n,t_n)),t) \toiso   a_k(S_1* \cdots * S_n, t_{t_1, \ldots, t_n})  $$ with $P=(((S_1,t_1), \ldots, (S_n,t_n)),t)$,  $k=\sum k_i$ and $$ \nu_X: X \toiso a_1((X), \bullet).  $$

There is a canonical functor (compare with \eqref{eq:UMoncat_MonCat})
\begin{equation}\label{eq:pseudo_moncat}
\mathbf{pseudo}\textrm{-}T\textrm{-}\mathbf{Alg}  \to \mathbf{MonCat}
\end{equation}
that assigns, to every pseudo-$T$-algebra $(\C,a, \theta, \nu)$, the underlying category $\C$ pro\-vid\-ed with the following monoidal structure: the monoidal product is given by the bifunctor $$ a_2: \C^{\times 2} \times \mathrm{Mag}_2(\bullet) \cong \C \times \C \to \C  $$ ($\mathrm{Mag}_2(\bullet)$ is the discrete one-object category). The associativity constraint is induced by $\theta$ and the unique isomorphism $p: (\bullet \bullet )\bullet \to \bullet (\bullet \bullet )$ in $\mathrm{Mag}_3(\bullet)$, more precisely this is
$$
\begin{tikzcd}[row sep= 0.5em, column sep=5em]
a_2((a_2((X,Y),\bullet \bullet),Z),\bullet \bullet)  \rar{a_2(\id, \nu_Z)}[swap]{\cong}  &  a_2((a_2((X,Y),\bullet \bullet),a_1((Z),\bullet )),\bullet \bullet) \\
\phantom{a_2((a_2((X,Y),\bullet \bullet),Z),\bullet \bullet)} \rar{\theta_{((X,Y),Z)}}[swap]{\cong } & a_3((X,Y,Z),(\bullet \bullet)\bullet) \phantom{-------}\\
\phantom{a_2((a_2((X,Y),\bullet \bullet),Z),\bullet \bullet) } \rar{a_3(\id_{(X,Y,Z)} \times p)}[swap]{\cong } & a_3((X,Y,Z),  \bullet(\bullet \bullet))\phantom{-------} \\
\phantom{a_2((a_2((X,Y),\bullet \bullet),Z),\bullet \bullet) } \rar{\theta^{-1}_{(X,(Y,Z))}}[swap]{\cong } &  a_2((a_1((X), \bullet) ,a_2((Y,Z),\bullet \bullet)),\bullet \bullet)\\
\phantom{a_2((a_2((X,Y),\bullet \bullet),Z),\bullet \bullet) } \rar{a_2(\nu_X^{-1},\id)}[swap]{\cong } &  a_2((X,a_2((Y,Z),\bullet \bullet)), \bullet \bullet) \phantom{----}
\end{tikzcd}
$$
(bullets have been removed from the subindices for clarity), and the left and right unit constraints are defined as in \eqref{eq:left_right_constr_pseudo} with the obvious changes.

The following theorem is essentially well-known:
\begin{theorem}\label{thm:equiv_pseudoTalg_magma}
The functor \eqref{eq:pseudo_moncat} is an equivalence of categories,
$$ \mathbf{pseudo}\textrm{-}T\textrm{-}\mathbf{Alg}  \simeq \mathbf{MonCat} .$$
This equivalence restricts to an equivalence between $\mathbf{str}\textrm{-}T\textrm{-}\mathbf{Alg} $ and the subcategory of strictly unital monoidal categores.
\end{theorem}

In contrast with \cref{thm:equiv_UMonCat_MonCat}, the functor from the previous theorem has a canonical quasi-inverse, namely if $\C$ is a monoidal category, then a quasi-inverse of \eqref{eq:pseudo_moncat} is given by assigning to $\C$ the pseudo-$T$-algebra with structure morphism $a$ given by
\begin{equation}
a(S,t) := Par(S,t)
\end{equation}
where $Par$ is as defined in \cref{subseq:non-strict}. It is easy to see that for the resulting pseudo-$T$-algebra, we have that $\theta= \id$.

\begin{remark}
It is in fact possible to construct a 2-monad $T$ whose category of strict $T$-algebras is equivalent to $\mathbf{MonCat}$, but it is much harder to describe explicitly since it arises as a colimit, see \cite[\S 5.5]{lack}
\end{remark}


Let us apply now Power's coherence \cref{thm:power} to the free magma 2-monad. For the same reason as in the \cref{subsec:strict_via_2monads}, $T$ preserves bijective-on-objects functors. Consider $\C$ a monoidal category, that we view as a pseudo-$T$-algebra via the equivalence of \cref{thm:equiv_pseudoTalg_magma}. Let us describe the resulting strict $T$-algebra $\mathcal{E}$.

As an ordinary category, $\mathcal{E}$ has the same objects as $T(\C)$, so pairs $(S,t)$, and morphisms
$$
\hom{\mathcal{E}}{S}{S'} = \hom{\mathcal{C}}{a(S)}{a(S')}= \hom{\mathcal{C}}{Par(S)}{Par(S')}.
$$
So, as an ordinary category, $\mathcal{E}$ coincides with the category  $\C_q$ constructed in \cref{subseq:non-strict}.

Using the same argument as in \cref{subsec:strict_via_2monads} mutatis mutandis, we obtain

\begin{theorem}
Let $\C$ be a monoidal category,  viewed as a pseudo-$T$-algebra via \cref{thm:equiv_pseudoTalg_magma}.  If $\mathcal{E}$ is the strict $T$-algebra obtained from \cref{thm:power}, then viewed as a monoidal category, $\mathcal{E}$ is exacty the non-strictification $\C_q$ of $\C$.

Furthermore, the pseudofunctor $a'': \mathcal{E} \to \C$ is precisely, up to these identifications, the monoidal equivalence $Par: \C_q \to \C$ from \cref{subseq:non-strict}.
\end{theorem}

\section{Strictification via bicategories}\label{sec:6}

Given a monoidal category $\C$, we explicitly constructed in \cref{sec:1} a strict monoidal category $\Cstr$ monoidally equivalent to $\C$. In this  section, we will first provide an alternative, explicit model for the strictification of a monoidal category, via the so-called category of right-module endofunctors of $\C$. Later, we will explain how this model can be regarded as a particular instance of a more general ``strictification'' result for bicategories (also known as weak 2-categories).


\subsection{Category of right-module endofunctors}\label{subsec:right-module_endofunctors}

The first construction we want to introduce, due to Joyal and Street \cite{JS}, can be viewed as a categorification of the following easy observation: if $M$ is a monoid, let us say that a set $X$ is a \textit{right $M$-module} if it is endowed with a right action of $M$. Given two right $M$-modules $X,Y$, a right $M$-module map $f: X \to Y$ is a set-theoretical map that preserves the right action of $M$, $f(x\cdot m)=f(x)\cdot m$. Let us write $\operatorname{End}_M(X)$ for the set of right $M$-module endomorphisms of $X$, which is itself a monoid with product $f \cdot g := f\circ g$ and with unit the identity map. If $X=M$ is viewed as a right $M$-module endowed with the right multiplication, then the canonical map
\begin{equation}\label{eq:monoid_iso}
M \to \operatorname{End}_M(M) \qquad , \qquad m \mapsto  m \cdot -
\end{equation}
is a monoid isomorphism, because for any right $M$-module map $f$, the equality $$ f(m)= f(1 \cdot m)= f(1) \cdot m  $$ expresses that $f$ is uniquely determined by the image of the unit element.

In the categorical setting, we will consider endofunctors $F: \C \to \C$ of a monoidal category $\C$. Requiring $F(X \otimes Y)= F(X) \otimes Y$ would be a too strong condition (as it is customary in category theory), so a more natural requirement is that $F(X \otimes Y)$ and $F(X) \otimes Y$ are related by a (fixed) natural isomorphism. 

Here is the concrete construction:

\begin{construction}\cite{JS}
Let $\C$ be a monoidal category. A \textit{right-module endofunctor} is a pair $(F, c)$ where $F: \C \to \C$ is a functor and $$c: \otimes \circ (F \times \id) \overset{\cong}{\Longrightarrow} F \circ \otimes$$ is a natural isomorphism such that the following diagram is commutative for all objects $X,Y,Z$:
\begin{equation}\label{eq:pentagon_End}
\begin{tikzpicture}[commutative diagrams/every diagram]
\node (P0) at (90:2.3cm) {$(F(X) \otimes Y) \otimes Z$};    
\node (P1) at (90+72:2cm) {$F(X \otimes Y)\otimes Z$} ;
\node (P2) at (90+2*72:2cm) {\makebox[3ex][r]{$F((X \otimes Y) \otimes Z)$}};
\node (P3) at (90+3*72:2cm) {\makebox[3ex][l]{$F(X \otimes (Y\otimes Z))$}};
\node (P4) at (90+4*72:2cm) {$F(X) \otimes (Y \otimes Z)$};
\path[commutative diagrams/.cd, every arrow, every label]
(P0) edge node[swap] {$  c_{X,Y} \otimes \id_Z $} (P1)
(P1) edge node[swap] {$c_{X\otimes Y, Z}$} (P2)
(P2) edge node {$F(a_{X,Y,Z})$} (P3)
(P4) edge node {$c_{X,Y\otimes Z}$} (P3)
(P0) edge node {$a_{FX,Y,Z}$} (P4);
\end{tikzpicture}
\end{equation}

A right-module endofuctor map $(F,c) \to (G,d)$ is a natural transformation $\alpha: F \Longrightarrow G$ with the following compatibility condition for all objects $X,Y$:
\begin{equation}\label{eq:square_End}
\begin{tikzcd}
F(X) \otimes Y \rar{c_{X,Y}} \dar[swap]{\alpha_X \otimes \id_Y} & F(X \otimes Y) \dar{\alpha_{X \otimes Y}} \\
G(X) \otimes Y \rar{d_{X,Y}} & G(X\otimes Y)
\end{tikzcd}
\end{equation}

Right-module endofunctors and right-module endomorphism maps form a category $\operatorname{End}_\C (\C)$. The composition of right-module endomorphism maps is given by vertical composition of natural transformations, and the identity of a right-module endomorphism map is the identity natural transformation. 

The category $\operatorname{End}_\C (\C)$ admits a strict monoidal structure: on objects, 
\begin{equation}\label{eq:mon_str_End1}
(F,c) \bullet (G,d) := (F \circ G, e) 
\end{equation} where $e_{X,Y}$ is defined as the composite
\begin{equation}\label{eq:mon_str_End2}
\begin{tikzcd}
FG(X) \otimes Y \rar{c_{GX,Y}} & F(G(X)\otimes Y) \rar{F(d_{X,Y})} & FG(X \otimes Y).
\end{tikzcd}
\end{equation}
On arrows, the tensor product is simply the horizontal composition of natural transformations, and the pair $(\id_\C, \id_{-\otimes-})$ serves as the unit of the monoidal structure.

We omit the tedious but straightforward details making sure that the above structure of strict monoidal category is well-defined, but the serious reader should check that (1) the composition of arrows in $\operatorname{End}_\C (\C)$ satisfies \eqref{eq:pentagon_End}, (2) the tensor product of objects satisfies  \eqref{eq:square_End}, (3) the tensor product of arrows satisfies \eqref{eq:pentagon_End}, (4) the tensor product is strictly associative on objects (this  is obvious for the endofunctor component, but it is not immediate for the natural isomorphism) and on arrows, (5) the tensor product is strictly unital.
\end{construction}

We are now ready to state the categorical counterpart of \eqref{eq:monoid_iso}: define
\begin{equation}
\Psi: \C \to \operatorname{End}_\C (\C) \quad , \quad \Psi(X):= (X \otimes -, a_{X,-,-}) \quad , \quad \Psi (f)= f\otimes \id_{(-)} .
\end{equation}
That $\Psi(X)$ is indeed a right-module endofunctor follows from the Pentagon axiom (and the naturality of the associativity constraint).
That $\Psi$ is a functor readily follows since any object $X$ in $\C$, $$ \Psi (g \circ f)_X = (g \circ f) \otimes \id_{X} = (g \otimes \id_X) \circ (f \otimes \id_X) = \Psi(g)_X \circ \Psi(f)_X,  $$ and besides $\Psi (\id_X)= \Psi (\id_{X} \otimes \id_{(-)}) = \id_{\Psi(X)}$. This functor can be upgraded to a strong monoidal functor: define a natural isomorphism $$\gamma_{X,Y}^\Psi= \gamma_{X,Y}: \Psi (X) \bullet \Psi (Y) \toiso \Psi (X\otimes Y),$$ that is, an isomorphism in $\operatorname{End}_\C (\C)$
$$ \gamma_{X,Y}: \Big( X \otimes (Y \otimes -), (\id_X \otimes a_{Y,-,-})\circ (a_{X,Y\otimes -, -}) \Big) \toiso \Big( (X \otimes Y)\otimes - , a_{X\otimes Y, -,-} \Big)  $$ natural in $X$ and $Y$ as
\begin{equation}\label{eq:gamma_End}
\gamma_{X,Y}:= a^{-1}_{X,Y,-}.
\end{equation}
The diagram \eqref{eq:square_End} becomes in this case the Pentagon equation, so $\gamma_{X,Y}$ is indeed an arrow in $\operatorname{End}_\C (\C)$. Lastly define $$  u^\Psi =u: (\id_\C, \id_{- \otimes -}) \to (\bm 1 \otimes - , a_{\bm 1, -,-}) $$ as
\begin{equation}\label{eq:u_End}
u:= \ell^{-1}.
\end{equation}
This is in fact a morphism in $\operatorname{End}_\C (\C)$ because \eqref{eq:square_End} translates into the left-hand side diagram of \eqref{eq:triangle2}. It is readily verified that the Hexagon axiom and \eqref{eq:4} become the Pentagon axiom, the Triangle axiom and \eqref{eq:triangle2}. Thence $\Psi$ is a strong monoidal functor.

\begin{theorem}[\cite{JS,EGNO}]\label{thm:equiv_right_mod_endo}
The functor
$$\Psi: \C \overset{\simeq}{\to} \operatorname{End}_\C (\C)$$ is a monoidal equivalence of categories.
\end{theorem}
\begin{proof}
We start by showing that $\Psi$ is faithful. Given arrows $f,g$ in $\C$, if $\Psi(f)=\Psi(g)$, then in particular $f \otimes \id_{\bm{1}}= g \otimes \id_{\bm{1}}$. This readily implies, by the naturality of the right unit constraint, that $f=g$.

Now we show that $\Psi$ is full. Given a right-module endofunctor map $ \alpha: \Psi (X) \to \Psi(Y)$, that is, a natural isomorphism $\alpha: X \otimes - \to Y \otimes -$ compatible with $a_{X,-,-}$ and $a_{Y,-,-}$ in the sense of \eqref{eq:square_End}, we claim that $\alpha = \Psi (f)$ where $f:=r_Y \circ \alpha_{\bm{1}} \circ r_X^{-1}: X \to Y$. Let us contemplate the following digram in $\C$, for any object $Z$ in $\C$:
$$
\begin{tikzcd}[row sep=scriptsize,column sep=scriptsize]
& X \otimes Z \arrow[from=dl,"r_X \otimes \id_Z"] \arrow[rr,equals] \arrow[dd, "\Psi(f)_Z",pos=0.2] & & X \otimes Z  \arrow[from=dl,"\id_Z \otimes \ell_Z"]\arrow[dd,swap, "\alpha_Z"] \\
(X \otimes \bm{1}) \otimes Z \arrow[rr,"a_{X,\bm{1},Z}",pos=0.3,  crossing over]\arrow[dd,swap,"\alpha_{\bm{1}}\otimes \id_Z"] & & X \otimes (\bm{1} \otimes Z )\\
& Y \otimes Z \arrow[from=dl, "r_Y \otimes \id_Z"]\arrow[rr, equals] &  & Y \otimes Z  \arrow[from=dl,swap, "\id_Y \otimes \ell_Z"] \\
(Y \otimes \bm{1}) \otimes Z \arrow[rr,"a_{Y,\bm{1},Z}"] & & Y \otimes (\bm{1} \otimes Z ) \arrow[from=uu, "\alpha_{\bm{1} \otimes Z}",pos=0.3, crossing over]
\end{tikzcd}
$$
In the above diagram the equals arrow represents the identity. Now, the front face of the cube commutes by \eqref{eq:square_End}, the left face of the cube by the definition of $\Psi(f)_Z= f \otimes \id_Z$, the right face of the cube by the naturality of $\alpha$, and the top and bottom faces by the Triangle axiom. Therefore the whole is commutative and and back face also commutes, that is, $\alpha= \Psi(f)$.

Next let us show that the functor $\Psi$ is  essentially surjective. Given a right-module endomorphism $(F,c)$, we claim that it is isomorphic to $\Psi (F(\bm{1}))$ in $\operatorname{End}_\C (\C)$. Indeed, consider the natural isomorphism of endofunctors of $\C$
\begin{equation}\label{eq:def_alpha_endC}
\alpha: F(\bm{1}) \otimes - \Longrightarrow F \qquad , \qquad \alpha_X := F(\ell_X) \circ c_{ \bm{1},X}. 
\end{equation}
To see that $\alpha$ is compatible with $a_{F \bm{1},-,-}$ and $c$ in the sense of \eqref{eq:square_End}, consider the following diagram:
$$
\begin{tikzcd}
(F(\bm{1}) \otimes X) \otimes Y \arrow{rr}{a_{\bm{1},X,Y}} \dar[swap]{c_{\bm{1},X} \otimes \id_Y} &  & F(\bm{1} \otimes (X \otimes Y) \dar{c_{\bm{1},X \otimes Y}} \\
F(\bm{1} \otimes X)\otimes Y \dar{F(\ell_X)\otimes \id_Y}\rar{c_{\bm{1}\otimes X,Y}} & F((\bm{1} \otimes X)\otimes Y)  \drar[swap]{F(\ell_X \otimes \id_Y)} \rar{F(a_{\bm{1},X,Y})} & F(\bm{1} \otimes (X\otimes Y)) \dar{F(\ell_{X\otimes Y})} \\
F(X) \otimes Y \arrow{rr}{c_{X,Y}} & & F(X\otimes Y)
\end{tikzcd}
$$
The top rectangle is nothing but \eqref{eq:pentagon_End} for $(F,c)$, the bottom left square commutes by the naturality of $c$, and the bottom right triangle by \eqref{eq:triangle2}. Therefore, the whole diagram commutes. By noting that the outer diagram is precisely the required diagram we conclude.

So far we have shown that $\Psi$ is an equivalence of categories. It only remains to demonstrate that $\Psi$ is a strong monoidal functor. 

For objects $X,Y$ in $\C$, first observe that
$$\Psi(X) \bullet \Psi (Y) = (X \otimes (Y \otimes -) , (\id_X \otimes a_{X,-,-})\circ a_{X,Y\otimes-,-})$$
and $$\Psi(X \otimes Y)= ((X \otimes Y)\otimes -, a_{X \otimes Y,-,-}).$$
Then define 
$$\gamma_{X,Y}: \Psi(X) \bullet \Psi (Y) \to \Psi(X \otimes Y)  $$ by $\gamma_{X,Y}:= a^{-1}_{X,Y,-}$, which is trivially a natural isomorphism. It is straightforward to see that the compatibility diagram \eqref{eq:square_End}  for $\gamma_{X,Y}$ amounts precisely to the Pentagon axiom in $\C$. 

Lastly, the arrow $$u: (\id_\C, \id_{- \otimes -}) \to \Psi(\bm{1})= (\bm{1} \otimes - , a_{\bm{1},-,-})$$ is defined as $u_X := \ell_X^{-1}$, another natural isomorphism. That \eqref{eq:square_End}  holds in this case follows from \eqref{eq:triangle2}.

The Hexagon axiom for $\Psi$ is a immediate consequence of the Pentagon axiom, and the commutativity of \eqref{eq:4} for $\Psi$ follows directly from the Triangle axiom and \eqref{eq:triangle2}. This concludes the proof.
\end{proof}

\begin{remark}
We warn the reader that the composite functor 
$$\Psi \circ Par : \Cstr \to \operatorname{End}_\C (\C) $$
is a monoidal equivalence of strict categories but it is \textit{not} a strict monoidal equivalence.
\end{remark}

\begin{corollary}\label{cor:diagram_CX1}
If $(F,c)$ is a right-module endofunctor, then the following diagram commutes for every object $X$ in $\C$:
$$
\begin{tikzcd}[column sep=0.25cm]
F(X) \otimes \bm{1} \drar[swap]{r_{FX}} \arrow{rr}{c_{X,\bm{1}}} &&  F(X \otimes \bm{1}) \dlar{F(r_X)} \\
 &     F(X) & 
\end{tikzcd}
$$
\end{corollary}
\begin{proof}
As we saw in the proof of \cref{thm:equiv_right_mod_endo}, the pair $(F,c)$ is isomorphic to $\Psi(F(\bm 1))$ in $\operatorname{End}_\C(\C)$ via the natural isomorphism $\alpha: F(\bm{1}) \otimes - \Longrightarrow F$ defined in \eqref{eq:def_alpha_endC}. Let us contemplate the following diagram in $\C$:
$$
\begin{tikzcd}
(F(\bm{1})\otimes X )\otimes \bm{1} \arrow[rrrr,"a_{F\bm{1},X, \bm{1}}"] \arrow[rrd,"r_{F\bm{1} \otimes X}"] \arrow[dd,"\alpha_X \otimes \id_{\bm{1}}"] &  &               &  & F(\bm{1})\otimes (X \otimes \bm{1}) \arrow[lld,"\id_{F\bm{1}}\otimes r_X"] \arrow[dd,"\alpha_{X \otimes \bm{1}}"] \\
                                       &  & F(\bm{1})\otimes X  &  &                           \\
F(X)\otimes  \bm{1} \arrow[rrrr,pos=0.3,"c_{X,\bm{1}}"] \arrow[rrd,"r_{FX}"]            &  &               &  & F(X\otimes \bm{1})\arrow[lld,"F(r_X)"]            \\
                                       &  & F(X)    \arrow[from=uu,"\alpha_X",pos=0.3,crossing over]        &  &                          
\end{tikzcd}
$$
The top face of the prism is precisely the right-hand side diagram of \eqref{eq:triangle2} so it commutes. Likewise the back face commutes  by \eqref{eq:square_End}, the left face by the naturality of the right unit constraint $r$ and the right side by the naturality of $\alpha$. So the whole diagram commutes and, in particular, so does the bottom face of the prism, which is the desired diagram.
\end{proof}

\subsection{Bicategories and their functors}

Next we would like to elaborate on a concept that can be viewed as joint generalizations of the notions of 2-category and monoidal category, namely that of bicategory.
Bicategories were first devised by Bénabou \cite{benabou}  and have been well-studied for the last sixty years \cite{maclane,gray,street,leinster,gurski,JY}. Rough\-ly speaking, a bicategory is the same thing as a 2-category except that the equalities of functors \eqref{eq:assoc_2-cat}--\eqref{eq:rightunit_2-cat}  do not hold ``on the nose'', but instead they are commutative up to a coherent isomorphisms, and these coherence constraints satisfy analogous Pentagon and Triangle axioms as the constraints for monoidal categories.

More precisely, a \textit{bicategory} $\B$  is the data of (compare with \cref{subsec:2-categories})
\begin{enumerate}
\item a family of objects $A, B, C, \ldots$ of $\B$,
\item for every pair of objects $A, B \in \B$, a hom-category $\B(A,B)$,
\item for each $A  \in \B$, a functor  $\id_{A}: * \to \B(A,A)$,
\item for every triple $A,B,C \in \B$, a functor called \textit{horizontal composition} $$\circ_{A,B,C} =\circ : \B(B,C) \times \B(A,B) \to \B(A,C),$$
\item for any quadruple of objects $A,B,C,D$ in $\B$, a natural isomorphism called the \textit{associator} $$a_{A,B,C,D}:  \circ_{A,B,D}(\circ_{B,C,D} \times \id_{\B(A,B)}) \overset{\cong}{\Longrightarrow} \circ_{A,C,D}(\id_{\B(C,D)} \times \circ_{A,B,C})  $$ of functors $$ \B(C,D) \times \B(B,C) \times \B(A,B) \to \B(A,D),$$
\item for any pair of objects $A,B$ in $\B$, natural isomorphisms called \textit{left} and \textit{right unitors}
\begin{align*}
\ell_{A,B}:& \circ_{A,B,B}(\id_{B} \times \id_{\B(A,B)}) \overset{\cong}{\Longrightarrow} \id_{\B(A,B)} \\
r_{A,B}:& \circ_{A,A,B}(\id_{\B(A,B)}\ \times \id_{A}) \overset{\cong}{\Longrightarrow} \id_{\B(A,B)} 
\end{align*}
\end{enumerate}
satisfying the following coherence conditions: for every $A,B,C,D,E \in \B$ and $f \in \B(A,B)$, $g \in \B(B,C)$, $h \in \B(C,D)$, $i \in \B(D,E)$, we have the following commutative diagrams in $\B(A,E)$ and $\B(A,C)$, respectively:
\begin{equation}\label{eq:pentagon_for_bicats}
\begin{tikzpicture}[commutative diagrams/every diagram]
\node (P0) at (90:2.3cm) {$((i \circ h)\circ g)\circ f$};    
\node (P1) at (90+72:2cm) {$(i \circ (h \circ g))\circ f$} ;
\node (P2) at (90+2*72:2cm) {\makebox[3ex][r]{$i\circ ((h\circ g) \circ f)$}};
\node (P3) at (90+3*72:2cm) {\makebox[3ex][l]{$i\circ (h\circ (g\circ f))$}};
\node (P4) at (90+4*72:2cm) {$(i\circ h)\circ (g\circ f)$};
\path[commutative diagrams/.cd, every arrow, every label]
(P0) edge node[swap] {$  a_{i,h,g} * \id_f $} (P1)
(P1) edge node[swap] {$a_{i, h \circ g, f}$} (P2)
(P2) edge node {$\id_i \times a_{h,g,f}$} (P3)
(P4) edge node {$a_{i,h, g * f}$} (P3)
(P0) edge node {$a_{i \circ h, g,f}$} (P4);
\end{tikzpicture}
\end{equation}
\begin{equation}\label{eq:triangle_for_bicats}
\begin{tikzcd}[column sep=0.3cm]
 (g \circ \id_B) \circ f \drar[swap]{r_g * \id_f} \arrow{rr}{a_{g, \id_B, f}} &&  g \circ (\id_g * \ell_f ) \dlar{\id_g\otimes\ell_{ f}} \\
 &     g \circ f & 
\end{tikzcd}
\end{equation}
In the diagrams above, we have suppressed from the notation the subscripts corresponding to the objects of $\B$, so for instance $a_{i,h,g} = (a_{B,C,D,E})_{i,h,g}$ in \eqref{eq:pentagon_for_bicats}. The terminology 0-cell, 1-cell and 2-cell will be used as for 2-categories.

\begin{examples}

This definition encodes, at once, the notions of (ordinary) category, 2-category, monoidal category and strict monoidal category:
\begin{enumerate}
\item A category is the same data as a locally discrete bicategory, that is, a bicategory $\B$ in which the categories $\B (A,B)$ are discrete, that is, contain only identities. Indeed in that case the associator and left and right unitors are necessarily the identity natural transformations, so they simply express that the composition is associative and left and right unital, and the diagrams \eqref{eq:pentagon_for_bicats} and \eqref{eq:triangle_for_bicats} are void.

\item A 2-category is the same data as a bicategory in which the associator and left and right unitors are the identity natural transformations. Indeed this last condition are indicating that the three equalities of functors \eqref{eq:assoc_2-cat}--\eqref{eq:rightunit_2-cat} hold ``on the nose''.

\item A monoidal category is the same data as a one-object bicategory.	Indeed if we denote by $\star$ such a single object, then $\C:= \B(\star,\star)$ is a category with monoidal product given by the horizontal composition, $f \otimes g := f \circ g$. In this case, \eqref{eq:pentagon_for_bicats} and \eqref{eq:triangle_for_bicats} are exactly the Pentagon and Triangle axioms for monoidal categories.

\item A strict monoidal category is a one-object 2-category, that is, a one-object bicategory where the associator and left and right unitors are the identity natural transformations. This follows directly from the previous cases.
\end{enumerate}
\end{examples}

Just like the notion of bicategory arises from allowing certain equalities in the definition of 2-category to hold only up to natural isomorphism, a \textit{lax functor} or morphism $F: \B \to \B'$ between bicategories is defined in a similar fashion as that of  a 2-functor with the only difference that the equalities of functors \eqref{eq:Anotherdiagram1} and \eqref{eq:Anotherdiagram2} hold only up to natural isomorphism.

Here is the definition: given bicategories $\B, \B'$, a \textit{lax functor} $F: \B  \to \B'$  is the data of (compare with \cref{subsec:2-functor})
\begin{enumerate}
\item for every object $A \in \B$, an object $F(A) \in \B'$,
\item for every pair of objects $A, B \in \B$, a functor$$F_{A,B}: \B(A,B) \to \B'(F(A),F(B))$$
\item for any triple of objects $A,B,C$ in $\B$, a natural transformation $$ \gamma_{A,B,C}^F= \gamma : \circ_{FA,FB,FC}(F_{B,C} \times F_{A,B}) \Longrightarrow F_{A,C} \circ_{A,B,C}   $$ of functors $$\B (B,C) \times \B(A,B) \to \B'(FA,FC)  $$
\item  for every object $A$ in $\B$, a natural transformation $$u_A^F= u: \id_{FA} \Longrightarrow F_{A,A} \id_{A}$$ of functors $ * \to \B' (FA,FA)$,
\end{enumerate}
such that for every quadruple $A,B,C,D \in \B$ and every $f \in \B(A,B)$, $g \in \B(B,C)$, $h \in \B(C,D)$, the following diagrams commute:

\begin{equation}\label{eq:hexagon_4_bicats}
\begin{tikzcd}[column sep={1cm,between origins}, row sep={1.732050808cm,between origins}]
    & {\makebox[3ex][r]{$(F(h) \circ F(g)) \circ F(f)$}} \arrow[rr, "a'_{Ff,Fg,Fh}"] \arrow[ld, "\gamma_{h,g} * \id_{Ff}"'] && {\makebox[3ex][l]{$F(h) \circ (F(g) \circ F(f))$}} \arrow[rd, "\id_{Fh} * \gamma_{g,f}"] &  \\
    F(h \circ g) \circ F(f) \arrow[rd, "\gamma_{h \circ g, f}"']&  &&  & F(h) \circ F(g \circ f) \arrow[dl,"\gamma_{h, g \circ f}"] \\
    & {\makebox[3ex][r]{$F((h \circ g) \circ f)$}} \arrow[rr, swap, "F(a_{f,g,h})"'] && {\makebox[3ex][l]{$F(h \circ (g \circ f))$}}  & 
\end{tikzcd}
\end{equation}
\begin{equation}\label{eq:squares_4_bicats}
\begin{tikzcd}[sep=2.7em]
\id_{FB} \circ F(f) \rar{\ell'_{Ff}} \dar[swap]{u * \id_{Ff}}  & F(f)  &  F(f) \circ \id_{FA}  \rar{r'_{Ff}} \dar[swap]{\id_{Ff} * u}  & F(f)  \\ 
F(\id_B) \circ F(f) \rar{\gamma_{\id_B,f}} & F(\id_B \circ f) \uar[swap]{F(\ell_f)} &   F(f)\circ F(\id_A) \rar{\gamma_{Ff, \id_A}} & F( f \circ \id_A)  \uar[swap]{F(r_f)}
\end{tikzcd}
\end{equation}
where we have suppressed  from the natural transformations the subscripts corresponding to objects of $\B$ to improve readability.

A lax functor between categories is called a \textit{pseudofuctor} (resp. a \textit{strict functor}) if the natural transformations $\gamma, u$ are natural isomorphisms (resp. the identity natural transformations).

\begin{examples}

It is readily seen that
\begin{enumerate}
\item A lax functor between locally discrete bicategories (that is, ordinary categories)  is the same data as an ordinary functor (in fact, such a functor is necessarily strict).
\item A strict functor between bicategories where the associators and unitors are the identities (that is, 2-categories) is the same data as a 2-functor.
\item A lax functor (resp. pseudofunctor, resp. strict functor) between one-object bicategories (that is, monoidal categories) is the same data as a lax (resp. strong, resp. strict) monoidal functor. Indeed \eqref{eq:hexagon_4_bicats} translate directly to the Hexagon axiom and \eqref{eq:squares_4_bicats} amounts exactly to \eqref{eq:4}.
\end{enumerate}
\end{examples}

Let us now define the bicategorical analogue of an equivalence of categories. We say that a pseudofunctor $F: \B \to \B'$ between bicategories is a \textit{biequivalence} if $F$ is
\begin{itemize}
\item \textit{a local equivalence}, in the sense that for every pair of objects $A,B$ of $\B$, the functor $$F_{A,B}: \B(A,B) \to \B'(F(A),F(B))$$ is an equivalence of categories (hence $F_{A,B}$ is fully faithful and essentially surjective in the usual sense),
\item \textit{essentially surjective up to equivalence}, in the sense that for every object $C$ in $\B'$, there exist an object $A$ in $\C$ and 1-cells $f \in \B'(FA, C)$, $g \in \B'(C , FA)$ such that $g\circ f$ is isomorphic to $\id_{FA}$ in $\B'(FA,FA)$ and $f \circ g$ is isomorphic to $\id_C$ in $\B'(C,C) $.
\end{itemize}

\begin{examples}
We have:
\begin{enumerate}
\item A biequivalence $F: \B \to \B'$ between locally discrete categories is the same data as an equivalence between the corresponding ordinary categories. Indeed $F$ being a local equivalence simply means that $F$, viewed as an ordinary functor, is fully faithful, because an equivalence between discrete categories must necessarily be an isomorphism, so we get a bijection on objects. Likewise $F$ being essentially surjective up to equivalence translates into $F$, as an ordinary functor, being essentially surjective because any 2-cell is the identity.
\item A biequivalence between 2-categories is called a \textit{2-equivalence}.
\item A biequivalence between one-object bicategories (a.k.a. monoidal categories) is the same data as a monoidal equivalence of monoidal categories. Indeed in this case the local equivalence condition is encoding the monoidal equivalence, since the essentially surjectivity up to equivalence  holds for any functor, taking $f,g$ the monoidal unit.
\end{enumerate}
\end{examples}

Now that we have the required language, let us give a more concise blueprint of what we aim to explain in this section and the reason why we are interested in bicategories. 

We showed in \cref{thm:1} that every monoidal category $\C$ is monoidally equivalent to a strict one $\Cstr$. It turns out that this result can be generalised to a statement about bicategories, that encodes the result when restricted to one-object bicategories. Namely, any bicategory $\B$ is biequivalent to a 2-category $\B^{\mathrm{str}}$, that is, a bicategory where the associator and the left and right unitors are the identities. Since biequivalence becomes monoidal equivalence and the associator and the unitors become the associativity and left and right unit constraints of a monoidal category, we get the result.  In fact, we will describe explicitly such a strictification of a bicategory and show that for one-object bicategories, one precisely obtains the category of right-module endofunctors that we studied in \cref{subsec:right-module_endofunctors}.

What is more,  it is possible to lift the 2-adjunction from  \cref{thm:2adj_str} defined by the strictification of monoidal categories to a diagram
\begin{equation}
\begin{tikzcd}[column sep={4em,between origins},row sep={6em,between origins}]
\mathsf{BiCat}
  \arrow[rrrr, bend left=10, swap, "{\mathrm{str}}"' pos=0.50]
&& \bot &&  \mathsf{2Cat} 
  \arrow[llll, bend left=10, swap, "U"' pos=0.50] \\
\mathsf{MonCat} \uar{U}
  \arrow[rrrr, bend left=10, swap, "{\mathrm{str}}"' pos=0.53]
&& \bot &&  \mathsf{strMonCat} \uar[swap]{U}
  \arrow[llll, bend left=10, swap, "U"' pos=0.47]
\end{tikzcd}
\end{equation}
where $\mathsf{BiCat}$ (resp. $\mathsf{2Cat}$) refer to the category of bicategories, pseudofunctors and transformations (resp. the category of 2-categories, 2-functors and 2-natural transformations). Here, transformations and 2-natural transformations refer to suitable notions of morphisms between pseudofunctors and 2-functors, respectively, and will be explained in the next subsection. Also,  we content ourselves regarding this as a  diagram of bicategories, but it can even be interpreted as a diagram of tricategories, see \cite{campbell}.

\subsection{Transformations and modifications}

Just like we have morphisms  between functors of ordinary categories, namely natural transformations, we can also consider morphisms between functors of bicategories, and these will be called transformations.

Here is the precise definition: if $F,G: \B \to \B'$ are lax functors between bicategories, a \textit{lax transformation} $\alpha: F \Longrightarrow  G $ is the data of 

\begin{enumerate}
\item for every object $A$ in $\B$, a 1-cell $\alpha_A \in \B'(FA,GA)$,
\item for every pair of objects $A,B$ in $\B$, a natural transformation $\alpha_{A,B}$ 
of the functors
\begin{equation}\label{eq:square_lax_trans}
\begin{tikzcd}[column sep=2.5em, row sep=3em]
\B(A,B) \rar{F_{A,B}} \dar[swap]{G_{A,B}} & \B'(FA, FB) \dar{(\alpha_B)_*} \\
\B' (GA,GB) \arrow[Rightarrow, shorten >=25pt, shorten <=25pt, ur,"\alpha_{A,B}"] \rar{(\alpha_A)^*} &\B'(FA,GB) 
\end{tikzcd}
\end{equation}
where $(\alpha_B)^*$ and $(\alpha_A)_*$ are the precomposition and postcomposition functors respectively,
\end{enumerate}
such that for every triple $A,B,C\in \B$ and every $f \in \B(A,B)$, $g \in \B(B,C)$, the following diagrams commute:
\begin{equation}\label{eq:octagon}
\begin{tikzcd}[column sep=2.5em, row sep=3em]
& (Gg \circ Gf) \circ \alpha_A \rar{\gamma^G_{g,f} * \id_{\alpha_A}} \dlar[swap]{a'_{Gg, Gf, \alpha_A}} & G(g \circ f) \circ \alpha_A \drar{\alpha_{g \circ f}}  & \\
Gg \circ (Gf \circ \alpha_A )  \dar{\id_{Gg}* \alpha_f} & & & \alpha_C \circ F(g \circ f)  \\
Gg \circ (\alpha_B \circ Ff) \drar{(a')^{-1}_{Gg, \alpha_B, Ff}} & & & \alpha_C \circ (Fg \circ Ff) \uar{\id_{\alpha_C}* \gamma^F_{g,f}}  \\
& (Gg \circ \alpha_B) \circ Ff \rar{\alpha_g * \id_{Ff}}  & (\alpha_C \circ Fg) \circ Ff \urar{a'_{\alpha_C, Fg, Ff}}  &	
\end{tikzcd}
\end{equation}
\begin{equation}\label{eq:pentagon_for_transformations}
\begin{tikzpicture}[commutative diagrams/every diagram]
\node (P0) at (90:2.3cm) {$\id_{GA} \circ \alpha_A$};    
\node (P1) at (90+72:2cm) {$\alpha_A$} ;
\node (P2) at (90+2*72:2cm) {\makebox[3ex][r]{$\alpha_A \circ \id_{FA}$}};
\node (P3) at (90+3*72:2cm) {\makebox[3ex][l]{$\alpha_A \circ F(\id_A)$}};
\node (P4) at (90+4*72:2cm) {$G(\id_A) \circ \alpha_A$};
\path[commutative diagrams/.cd, every arrow, every label]
(P0) edge node[swap] {$  \ell'$} (P1)
(P1) edge node[swap] {$(r')^{-1}$} (P2)
(P2) edge node {$\id_{\alpha_A}*u^F$} (P3)
(P4) edge node {$\alpha_{\id_A}$} (P3)
(P0) edge node {$u^G *\id_{\alpha_A}$} (P4);
\end{tikzpicture}
\end{equation}

A lax transformation in which every component $(\alpha_{A,B})_f$ is an isomorphism (resp. the identity) is called a \textit{strong} (resp. \textit{strict}) transformation.

\begin{examples}
We continue the previous examples:
\begin{enumerate}
\item If $F,G: \B \to \B'$ are strict functors between locally discrete categories, i.e. functors between ordinary categories, then a lax transformation (which in this case is necessarily strict) is the same data as a natural transformation. Indeed in this case \eqref{eq:octagon} and \eqref{eq:pentagon_for_transformations} are void.
\item If $F,G: \B \to \B'$ are 2-functors between 2-categories (viewed as bicategories), then a strict transformation is called a \textit{2-natural transformation}.
\end{enumerate}
\end{examples}

\begin{remark}
Monoidal natural transformations between monoidal functors are however not recovered from lax transformations, but from \textit{oplax transformations}, resulting from considering the natural transformations $\alpha_{A,B}$ from \eqref{eq:square_lax_trans}  with the opposite direction \cite{JY}.
\end{remark}

In a bicategory (for example a 2-category), having 2-cells allows us to consider morphisms between transformations (for example 2-natural transformations). These morphisms are called modifications.

More precisely, if $F,G: \B \to \B'$ are lax functors between bicategories and $\alpha, \beta: F \Longrightarrow G$ are lax transformations, a \textit{modification} $\Gamma: \alpha \LLongrightarrow \beta$ is the data of
\begin{enumerate}
\item for every object $A$ in $\B$, a 2-cell  $$\Gamma_A: \alpha_A \to \beta_A$$ 
(this is an arrow in the category $\B'(FA,GA)$)
\end{enumerate}
satisfying the following compatibility condition: for every pair of objects $A,B$ in $\B$ and every 1-cell $f \in \B(A,B)$, the diagram
\begin{equation}\label{eq:square_modifications}
\begin{tikzcd}[column sep=3.5em, row sep=2.5em]
Gf \circ  \alpha_A \rar{\id_{Gf}* \Gamma_A} \dar[swap]{\alpha_f} & Gf \circ \beta_A \dar{\beta_f} \\
\alpha_B \circ Ff \rar{\Gamma_B *\id_{Ff}} & \beta_B \circ Ff
\end{tikzcd}
\end{equation}
commutes. If every 2-cell $\Gamma_A: \alpha_A \to \beta_A$ is invertible, then we say that the modification is \textit{invertible}.

Given pseudofunctors $F,G: \B \to \B'$, let us define the category $\mathsf{Str}(F,G)$ as follows: its objects are strong transformations $\alpha: F \Longrightarrow G$, and its arrows are modifications $\Gamma: \alpha \LLongrightarrow \beta$. The composition of two modifications  and $\Gamma: \alpha \LLongrightarrow \beta$ and  $\Delta: \beta \LLongrightarrow \gamma$ is given by $$(\Delta\circ \Gamma)_A:= \Delta_A \circ \Gamma_A : \alpha_A\to \gamma_A$$ (this is a composition of arrows in the category $\B'(FA,GA)$). The identity of a strong transformation $\alpha$ is given by the identity modification $(\id_\alpha)_A := \id_{\alpha_A}: \alpha_A \to \alpha_A$, that is, the identity of the object $\alpha_A \in \B'(FA,GA)$. It is readily seen, pasting the corresponding diagrams \eqref{eq:square_modifications}, that this is indeed a category.

\subsection{Strictification of bicategories}

Let us move on to explain how to construct a 2-category biequivalent to a given bicategory. The construction will be based on the bicategorical Yoneda lemma.

Recall that if $\C$ is a (locally small) ordinary category and $X$ is an object in $\C$, then the \textit{Yoneda embedding}
$$ Y: \C \hooklongrightarrow  \mathbf{Set}^{\C^{\mathrm{op}}} \qquad , \qquad X \mapsto \homC{-}{X}$$ induces an equivalence between $\C$ and a  full subcategory of a presheaf category (namely, its essential image). We will see that this statement can be categorified into a statement of bicategories.

\begin{construction}[e.g. \cite{JY}]
Let $\B$ be a bicategory. The \textit{opposite bicategory} is the bicategory $\B^{\mathrm{op}}$ determined by the following data: the objects of $\B^{\mathrm{op}}$ are the same as the objects of $\B$. The hom categories are given by $\B^{\mathrm{op}} (A,B) := \B(B,A)$. The identity of an object $A$ in $\B^{\mathrm{op}}$ is the same as the identity of $A$ in $\B$. The horizontal composition is the functor $\circ_{A,B,C}^{\mathrm{op}}:= \circ_{C,B,A} \circ \mathrm{twist}$. The associator is given by $a^{\mathrm{op}}_{h,g,f}:= a^{-1}_{f,g,h}$ and the left and right unitors by $\ell_f^{\mathrm{op}} := r^{-1}_f$ and $r^{\mathrm{op}}_f:= \ell_f^{-1}$. It is immediate to check that \eqref{eq:pentagon_for_bicats} and \eqref{eq:triangle_for_bicats} hold so $\B^{\mathrm{op}}$ is indeed a bicategory.
\end{construction}

\begin{construction}[e.g. \cite{JY}]

Now let us consider the 2-category $\mathsf{Cat}$  of categories, functors and natural transformations. If $\B$ is a bicategory,  let us define the presheaf 2-category $\mathsf{Cat}^{\B^{\mathrm{op}}}$ as follows: its objects are pseudofunctors $\B^{\mathrm{op}} \to \mathsf{Cat}$. The hom categories are given by strong transformations and modifications between them, $$\mathsf{Cat}^{\B^{\mathrm{op}}}(F,G):= \mathsf{Str}(F,G).  $$ Given a pseudofunctor $F$, its identity transformation $\id_F: F \to F$ is the strict transformation given by the 1-cell $(\id_F)_{A} := \id_{FA} \in \mathsf{Cat}(FA,FA)$ for every object $A$ in $\B$, and the natural isomorphism
$$ ((\id_F)_{A,B})_f := \id_{Ff}$$
for every pair of objects $A,B$ in $\B$ and every $f \in \B(A,B)$ (note that this is an arrow of the category $\mathsf{Cat}(FA,FA)$, that is the identity natural transformation). It is easy to check that \eqref{eq:octagon} and \eqref{eq:pentagon_for_transformations} hold. Lastly, given pseudofunctors  $F,G,H: \B^{\mathrm{op}} \to \mathsf{Cat} $, the \textit{horizontal composition} 
$$ \circ : \mathsf{Str}(G,H) \times \mathsf{Str}(F,G) \to \mathsf{Str}(F,H)$$ is defined as follows: given strong transformations $\alpha: F \Longrightarrow G$ and $\beta: G \Longrightarrow H$, then setting
\begin{equation}\label{eq:hor_comp_Str1}
(\beta \circ \alpha)_A:=  \beta_A \circ \alpha_A 
\end{equation} (this is an object in $\mathsf{Cat}(FA,HA)$) and
\begin{equation}\label{eq:hor_comp_Str2}
((\beta \circ \alpha)_{A,B})_f:= (\id_{\beta_B}* \alpha_f)  \circ  (\beta_f*\id_{\alpha_A}) : Hf \circ \beta_A \circ \alpha_A \to \beta_B \circ \alpha_B \circ Ff 
\end{equation}
 for objects $A,B$ in $\B$ and $f \in \B^{\mathrm{op}} (A,B)= \B (B,A)$. On the other hand, for strong transformations $\alpha, \alpha ' : F \Longrightarrow G$ and $\beta, \beta' : G \Longrightarrow H$ and modifications $\Gamma: \alpha \LLongrightarrow \alpha'$ and $\Delta: \beta \LLongrightarrow \beta'$, the horizontal composition is defined as $$(\Delta * \Gamma)_A := \Delta_A * \Gamma_A : (\beta \circ \alpha)_A \to (\beta' \circ \alpha')_A.$$ Because $\mathsf{Cat}$ is a 2-category, it is straightforward to check that $\beta \circ \alpha$ is a strong transformation and $\Delta * \Gamma$ is a modification, as well as the fact that the horizontal composition is indeed associative and unital.
\end{construction}

The pseudofunctors $\B^{\mathrm{op}} \to \mathsf{Cat}$ that we are interested in are, just like for the Yoneda lemma, the representable ones:

\begin{construction}[e.g. \cite{JY}]\label{const:representable}
Let $\B$ be a bicategory and $E$ an object of $\B$. We define the \textit{pseudofunctor represented by $E$} $$\B(-,E): \B^{\mathrm{op}} \to \mathsf{Cat}$$ as follows: it assigns, for every object $A$ in $\B$, the hom category $\B(A,E)$. For objects $A,B$ in $\B$, the functor $$ \B(-,E)_{A,B}: \B (A,B) \to \mathsf{Cat}(\B (B,E), \B(A,E))  $$ sends a 1-cell $f \in \B (A,B)$ to the precomposition functor
\begin{equation}\label{eq:precomp_functor}
f^* := - \circ_{A,B,E} f : \B(B,E) \to \B (A,E),
\end{equation}
and a 2-cell  $\theta: f \to g$ in $\B (A,B)$ to the natural transformation $\theta^* : f^* \Longrightarrow g^*$ of functors $\B(B,E) \to \B(A,E)$ determined by  $\theta^*_h := \id_h * \theta : h \circ f \to h \circ g $ for every $h \in \B(B,E)$.

For objects $A,B,C$ in $\B$, the constraint $\gamma^{ \B(-,E)}_{A,B,C}=\gamma$ is defined as follows: for 1-cells $f \in \B(A,B)$ and $g \in \B(B,C)$, the natural transformation  $$\gamma_{g,f} : g^* \circ f^* \Longrightarrow (g \circ f)^*$$ of functors $ \B(C,E) \to \B(A,E)  $ given by the 2-cell
\begin{equation}\label{eq:def_gamma_gfh}
(\gamma_{g,f})_h := a_{h,f,g}: (g^* \circ f^*)(h)= (h \circ f) \circ g \to  h \circ (f \circ g)=(f \circ g)^*(h) 
\end{equation}
 in $\B(A,E)$,  for $h \in \B(C,E)$.

Lastly, for an object $A$ in $\C$, the constraint $u^{ \B(-,E)}_{A,}=u$  is set to be the natural transformation $u: \id_{\B(A,E)} \Longrightarrow \id_A^*$ of functors $\B(A,E)\to \B(A,E)$ defined by $$u_h:= r_h^{-1}: h \to h \circ \id_A = \id_A^*(h)$$ for $h \in \B(A,E)$.

That \eqref{eq:hexagon_4_bicats} holds follows from the pentagon axiom \eqref{eq:pentagon_for_bicats} for $\B$, and that \eqref{eq:squares_4_bicats} hold is consequence of \eqref{eq:triangle_for_bicats} and the analogue of \eqref{eq:triangle2} for bicategories. For a detailed argument see \cite[\S 4.5]{JY}
\end{construction}


Let us move on to construct the bicategorical analogue of the Yoneda embedding.

\begin{construction}[e.g. \cite{JY}]
Let $\B$ be a bicategory. We aim to construct a pseudofunctor  $$Y: \B \hooklongrightarrow   \mathsf{Cat}^{\B^{\mathrm{op}}} $$ from $\B$ to the presheaf 2-category $\mathsf{Cat}^{\B^{\mathrm{op}}}$ in the following way: to every object $A$ in $\B$ is assigned the representable pseudofunctor $Y_C:= \B(-,C)$ from \cref{const:representable}. Given a pair of objects $C,D$ in $\B$,  $$Y_{C,D}: \B(C,D) \to \mathsf{Str}(Y_C, Y_C)$$ is the functor defined in the following way:
\begin{itemize}
\item $Y_{C,D}$ sends a 1-cell $f \in \B(C,D)$ to the strong transformation $Y_{C,D}(f)$ that assigns to every object $B$ in $\B$ the \textit{postcomposition functor }
\begin{equation}\label{eq:post_comp}
Y_{C,D}(f)_B=f_* := f \circ_{B,C,D} -: \B(B,C) \to \B(B,D) 
\end{equation}
and for every pair of objects $B,A$ in $\B$, a natural transformation $$ Y_{C,D}(f)_{B,A}: (f_*)_* \circ (Y_D)_{B,A} \Longrightarrow (f_*)^*  \circ (Y_C)_{B,A}$$ of functors $\B^{\mathrm{op}}(B,A)= \B(A,B) \to   \mathsf{Cat}(\B (B,C), \B(A,D))  $ given componentwise, for every $g \in \B(A,B)$, by the natural transformation $$ (Y_{C,D}(f)_{B,A})_g: f_* \circ g^* \Longrightarrow g^* \circ f_*  $$ of functors $\B(B,C) \to \B(A,D)$ determined by 
\begin{equation}\label{eq:a_inverse}
((Y_{C,D}(f)_{B,A})_g)_h:= a_{f,h,g}: (f\circ h)\circ g \to f \circ (h\circ g)
\end{equation}
for every $h \in \B(B,C)$.
\item $Y_{C,D}$ sends a 2-cell $\alpha: f \to g$ to the modification $$Y_{C,D}(\alpha): Y_{C,D}(f)\LLongrightarrow Y_{C,D}(g)$$ that assigns to every object $B$ the natural transformation $$ Y_{C,D}(\alpha)_B:= \alpha_* : f_* \Longrightarrow g_*$$ of functors $\B(B,C) \to \B(B,D)$ given componentwise by
\begin{equation}\label{eq:Ystarstar_in_arrows}
(\alpha_*)_h:= \alpha * \id_h : f \circ h \to g \circ h
\end{equation}
for every $h \in \B(B,C)$.
\end{itemize}
Given objects $A,B,C$ in $\B$, the constraint $\gamma^Y_{A,B,C}=\gamma$ of the pseudofunctor $Y$ is the natural isomorphism
$$ \gamma: \circ_{Y_A,Y_B,Y_C}(Y_{B,C}\times Y_{A,B})  \Longrightarrow   Y_{A,C}\circ_{A,B,C}  $$
of functors $\B(B,C)\times \B(A,C) \to \mathsf{Str}(Y_A,Y_B)$ is set to be, for 1-cells $f \in \B(A,B)$, $g \in \B(B,C)$, componentwise the modification $$ \gamma_{g,f} : Y_{B,C}(g) \circ Y_{A,B}(g) \LLongrightarrow Y_{A,C}(g \circ f) $$ given by, for every object $E$ in $\B$, the natural transformation $$ (\gamma_{g,f})_E: g_* \circ f_* \Longrightarrow (g \circ f)_* $$ of functors $\B(E,A)\to \B(E,C)$ determined by
\begin{equation}\label{eq:gamma_bicat}
((\gamma_{g,f})_E)_h := a^{-1}_{g,f,h}: (g_* \circ f_*)(h)= g \circ (f \circ h) \to (g \circ f) \circ h = (g \circ f)_*(h)
\end{equation}
for $h \in \B(E,A)$. Similarly, for $A$ in $\B$, the constraint $u_A^Y=u$ is determined by 
\begin{equation}\label{eq:u_bicat}
(u_E)_h := (\ell^{-1}_{E,A})_h: h \to \id_A\circ h  
\end{equation}
 for an object $E$ in $\B$ and $h \in \B(E,A)$. 
\end{construction}

We can finally state the bicategorical Yoneda lemma:

\begin{theorem}[\cite{street,leinster_hohc,JY,bakovic}]\label{thm:bicategorical_Yoneda}
For any bicategory $\B$, the assignment $$ Y: \B \hooklongrightarrow   \mathsf{Cat}^{\B^{\mathrm{op}}} \qquad , \qquad C \mapsto  Y_C:= \B(-,C)$$ constructed above gives rise to a well-defined pseudofunctor which is a local equivalence, $$Y_{A,B}: \B(A,B) \overset{\simeq}{\to} \mathsf{Str}(Y_A, Y_B).$$
\end{theorem}

This functor is called the \textit{Yoneda pseudofunctor}. Since $ \mathsf{Cat}^{\B^{\mathrm{op}}}$ is a 2-category, the strictification result for bicategories will be an immediate consequence of the bicategorical Yonda lemma. More precisely, let $\mathsf{str}(\B)$ be the essential image of $Y$, that is, the full sub-2-category on objects the represented pseudofunctors $\B(-,C)$ for $C$ in $\B$. Therefore we obtain
\begin{corollary}
Any bicategory $\B$ is biequivalent to a 2-category $\mathsf{str}(\B)$ through the Yoneda pseudofunctor $Y: \B \to \mathsf{str}(\B)$.
\end{corollary}

If $\B$ is a one-object bicategory, then $\mathsf{str}(\B)$  is a one-object 2-category, that is, a strict monoidal category, and  we obtain

\begin{corollary}\label{cor:Y_star_star}
Let $\C$ be a monoidal category, viewed equivalently as a one-object bicategory $\B$ with single object $\star$. Then $\C$ is monoidally equivalent, via the Yoneda pseudofunctor,  to a strict monoidal category,
$$ Y_{\star, \star}: \B(\star, \star) = \C \overset{\simeq}{\to}  \mathsf{Str}(Y_\star, Y_\star), $$ where $Y_\star = \B(-, \star)$.
\end{corollary}

Let us unravel the construction of the strict monoidal category $\mathsf{Str}(Y_\star, Y_\star)$ as well as the  monoidal equivalence $Y_{\star, \star}:  \C \overset{\simeq}{\to}  \mathsf{Str}(Y_\star, Y_\star) $ induced by the Yoneda pseudofunctor from the previous corollary.

An object in $\mathsf{Str}(Y_\star, Y_\star)$ is a strong transformation $\alpha: Y_\star \to Y_\star$. By definition, this is the data of
\begin{itemize}
\item a single 1-cell $\alpha_\star =: F \in \mathsf{Cat}(\C, \C)$  (since $\B$ has one single object), that is, an endofunctor $F: \C \to \C$,
\item a natural isomorphism $\alpha_{\star, \star}$ of functors
$$
\begin{tikzcd}[column sep=2.5em, row sep=3em]
\C \rar{(Y_\star)_{\star, \star}} \dar[swap]{(Y_\star)_{\star, \star}} & \mathsf{Cat}(\C, \C) \dar{F_*} \\
\mathsf{Cat}(\C, \C) \arrow[Rightarrow, shorten >=18pt, shorten <=18pt, ur,"\alpha_{\star, \star}"] \rar{F^*} & \mathsf{Cat}(\C, \C)
\end{tikzcd}
$$
By \eqref{eq:precomp_functor},  $(Y_\star)_{\star, \star}$ is the functor $(Y_\star)_{\star, \star}(X)= - \otimes X$ for $X \in \C$, that we write  $(Y_\star)_{\star, \star}= - \otimes ?$. Therefore,  $\alpha_{\star, \star}$ is a natural isomorphism $$ \alpha_{\star, \star}: F(-)\otimes ? \overset{\cong}{\Longrightarrow} F(- \otimes ?)  $$ of functors $\C \to \mathsf{Cat}(\C, \C)$. By the exponential law for categories (provided that $\mathbf{Cat}$ is cartesian closed), giving such functors amounts to giving functors $\C \times \C \to \C$, and giving $\alpha_{\star, \star}$ amounts to giving a natural isomorphism  $$ c: F(-)\otimes - \overset{\cong}{\Longrightarrow} F(- \otimes -)  .$$ Both natural isomorphims are related by
\begin{equation}\label{eq:equiv_alpha_c}
c_{X,Y}=((\alpha_{\star, \star})_Y)_X.
\end{equation}
\end{itemize}
In other words, $\alpha= (F,c)$. Let us now analyse the coherence diagrams \eqref{eq:octagon} and \eqref{eq:pentagon_for_transformations}. For objects $Y,Z \in \C$ (that is, for 1-cells in $\B(\star, \star)$),  the octagon \eqref{eq:octagon} becomes a pentagon since $\mathsf{Cat}$ is a 2-category and then the associator $a'$ is the identity. Keeping in mind the definitions of \cref{const:representable}, the octagon diagram translates in this case to the following diagram of functors and natural isomorphisms between them:
\begin{equation*}
\begin{tikzpicture}[commutative diagrams/every diagram]
\node (P0) at (90:2.3cm) {$(F(-) \otimes Y) \otimes Z$};    
\node (P1) at (90+72:2cm) {$F(- \otimes Y)\otimes Z$} ;
\node (P2) at (90+2*72:2cm) {\makebox[3ex][r]{$F((- \otimes Y) \otimes Z)$}};
\node (P3) at (90+3*72:2cm) {\makebox[3ex][l]{$F(- \otimes (Y\otimes Z))$}};
\node (P4) at (90+4*72:2cm) {$F(-) \otimes (Y \otimes Z)$};
\path[commutative diagrams/.cd, every arrow, every label]
(P0) edge[double, -implies] node[swap] {$  (\alpha_{\star,\star})_Y \otimes \id_Z $} (P1)
(P1) edge[double, -implies] node[swap] {$(\alpha_{\star, \star})_{Z}*\id_{-\otimes Y}$} (P2)
(P2) edge[double, -implies] node {$F(a_{-,Y,Z})$} (P3)
(P4) edge[double, -implies] node {$(\alpha_{\star, \star})_{Y\otimes Z}$} (P3)
(P0) edge[double, -implies] node {$a_{F(-),Y,Z}$} (P4);
\end{tikzpicture}
\end{equation*}
Evaluating at an object $X$ in $\C$, we get exactly \eqref{eq:pentagon_End}. Therefore, the pair $(F,c)$ obtained from $\alpha$ is a right-module endofunctor of $\C$. On the other hand, \eqref{eq:pentagon_for_transformations} in this case reads
$$
\begin{tikzcd}[column sep=0.25cm]
F(-) \otimes \bm{1} \arrow[Rightarrow,from=dr,"r^{-1}F"] \arrow[Rightarrow]{rr}{(\alpha_{\star, \star})_{\bm 1}} &&  F(- \otimes \bm{1}) \arrow[Rightarrow,,from=dl,swap,"Fr^{-1}"] \\
 &     F & 
\end{tikzcd}
$$
where $Fr^{-1}$ and $r^{-1}F$ denote the whiskerings of $F$ and the inverse of the right unit constraint. Evaluating this diagram on an arbitrary object $X$ in $\C$ yields the diagram from \cref{cor:diagram_CX1}. Therefore, for one-object bicategories, the axiom \eqref{eq:pentagon_for_transformations} is redundant. The upshot of this discussion is that an object of $\mathsf{Str}(Y_\star,Y_\star)$ is precisely a right-module endofunctor of $\C$.

Let us now consider the arrows of $\mathsf{Str}(Y_\star,Y_\star)$. Given objects $\alpha=(F,c)$ and $\beta=(G,d)$ in $\mathsf{Str}(Y_\star,Y_\star)$, a modification $\Gamma: \alpha \LLongrightarrow \beta$ is, by definition, the data of a single 2-cell $\Gamma_\star: \alpha_\star  \to  \beta_\star$ (since $\B$ has one single object), that is, a natural transformation $\Gamma_\star : F \Longrightarrow G$.  The  compatibility condition \eqref{eq:square_modifications} translates in this case to the following commutative diagram of functors and natural transformations,  for every object $Y \in \C$ (i.e. a 1-cell in $\B (\star, \star))$:
$$
\begin{tikzcd}[arrows=Rightarrow, column sep=3.5em, row sep=3em]
F(-)\otimes Y \rar{\Gamma_\star \otimes \id_{Y}} \dar[swap]{(\alpha_{\star, \star})_Y} &  G(-)\otimes Y \dar{(\beta_{\star, \star})_Y} \\
F(- \otimes Y) \rar{\Gamma_\star * \id_{- \otimes Y}} & G(- \otimes Y)
\end{tikzcd}
$$
Evaluating at an object $X$ in $\C$, we get back exactly \eqref{eq:square_End}. Whence $\mathsf{Str}(Y_\star, Y_\star)= \operatorname{End}_\C (\C)$ as an ordinary category. What is more, the monoidal structures of both categories also coincide. Indeed since both monoidal structures are strict it suffices to check the monoidal products and the units. Now if $\alpha=(G,d)$ and $\alpha=(F,c)$, \eqref{eq:hor_comp_Str1} directly translates to the first slot of \eqref{eq:mon_str_End1}, that is $\beta_\star \circ \alpha_\star= F \circ G$. On the other hand \eqref{eq:hor_comp_Str1} translates to
\begin{align*}
((\beta \circ \alpha)_{\star,\star})_Y &= (\id_F * (\alpha_{\star, \star})_Y ) \circ ((\beta_{\star, \star})_Y * \id_G) \\
&= Fd_{-,Y} \circ c_{G(-),Y},
\end{align*}
and evaluating at an object $X$ in $\C$ we get back \eqref{eq:mon_str_End2}.

Let us now move on to unravel the  monoidal equivalence $$Y_{\star, \star}: \B(\star, \star)= \C \overset{\simeq}{\to}  \operatorname{End}_\C (\C)=  \mathsf{Str}(Y_\star, Y_\star) $$ induced by the Yoneda pseudofunctor. Given an object $X$ in $\C$ (i.e. a 1-cell in $\B(\star, \star)$), we have that $Y_{\star, \star}(X)$ consists, according to  \eqref{eq:post_comp} and \eqref{eq:a_inverse}, of the functor $$Y_{\star, \star}(X)_\star =   X \otimes - : \C \to \C$$ together with the natural transformation given by $$((Y_{\star, \star}(X)_\star)_{\star, \star})_{Z})_Y =a_{X, Y,Z}: (X \otimes Y)\otimes Z \to X \otimes (Y \otimes Z) . $$  
Viewing $Y_{\star, \star}(X)_\star$ as a natural transformation $c$ of functors $\C \times \C \to \C$ according to \eqref{eq:equiv_alpha_c}, we conclude that  $$c= a_{X,-,-}.$$
Moreover, if $f: X \to Y$ is an arrow in $\C$, then we have, according to \eqref{eq:Ystarstar_in_arrows}, that $$(Y_{\star, \star}(X)_\star (f)= \id_X \otimes f  $$
Therefore, we conclude that $Y_{\star, \star} =\Psi$ as ordinary functors. Let us finally check that they coincide as strong monoidal functors. The monoidal constraint $\gamma^{Y_{\star, \star}}= \gamma$ for $Y_{\star, \star}$ is, according to \eqref{eq:gamma_bicat}, equal to $\gamma_{X,Y}= a^{-1}_{X,Y,-}$ for objects $X,Y$ in $\C$, so it coincides with \eqref{eq:gamma_End}. Likewise, the unital constraint $u^{Y_{\star, \star}}=u$ equals $\ell^{-1}$ according to \eqref{eq:u_bicat}, whence it coincides with \eqref{eq:u_End}.

The upshot of the previous discussion is the following

\begin{theorem}
If $\B$ is a one-object bicategory with underlying monoidal category $\C:= \B(\star,\star)$, then the underlying strict monoidal category $\mathsf{Str}(Y_\star, Y_\star) $ of the 2-category $\mathsf{str}(\B)$ is precisely the category of right-module endofunctors $\operatorname{End}_\C(\C)$.

Furthermore, the monoidal equivalence $$Y_{\star, \star}:  \C \overset{\simeq}{\to}  \mathsf{Str}(Y_\star, Y_\star)$$ from \cref{cor:Y_star_star} is exactly the monoidal equivalence $$\Psi: \C \overset{\simeq}{\to} \operatorname{End}_\C (\C)$$ from \cref{thm:equiv_right_mod_endo}.
\end{theorem}

\bibliographystyle{halpha-abbrv}
\bibliography{bibliografia}

\end{document}